\documentclass[aop]{imsart}

\RequirePackage{amsthm,amsmath,amsfonts,amssymb}
\RequirePackage[numbers]{natbib}
\RequirePackage[colorlinks,citecolor=blue,urlcolor=blue]{hyperref}
\RequirePackage{graphicx}
 \usepackage{amssymb}
 \usepackage{extpfeil}
 \usepackage{enumerate,enumitem}

\usepackage{hyperref}
\usepackage{xcolor}
\usepackage[capitalize]{cleveref}
\usepackage{enumerate}

\usepackage{tikz}
\usetikzlibrary{decorations.pathreplacing,shapes.misc}

\startlocaldefs

\numberwithin{equation}{section}
\newtheorem{theorem}{Theorem}[section]
\newtheorem{con}{Con}[section]
\newtheorem{lemma}[theorem]{Lemma}
\newtheorem{proposition}[theorem]{Proposition}
\newtheorem{corollary}[theorem]{Corollary}
\theoremstyle{remark}
\newtheorem{remark}[theorem]{Remark}
\theoremstyle{definition}
\newtheorem{definition}[theorem]{Definition}

\newtheorem{property}[con]{Property}
\newtheorem{assumption}[theorem]{Assumption}
\Crefname{assumption}{Assumption}{Assumptions}
\Crefname{property}{Property}{Properties}

\newtheorem{conjecture}{Conjecture}
\newtheorem{innercustomthm}{Theorem}
\newenvironment{customthm}[1]
  {\renewcommand\theinnercustomthm{#1}\innercustomthm}
  {\endinnercustomthm}
  \theoremstyle{remark}

\newcommand{\R}{\mathbb{R}}
\newcommand{\T}{\mathbb{T}}
\newcommand{\N}{\mathbb{N}}
\newcommand{\Z}{\mathbb{Z}}
\renewcommand{\P}{{\mathbb P}}
\newcommand{\sym}{\mathrm{\mathrm{sym}}}
\newcommand{\ds}{\displaystyle}
\newcommand{\eps}{\varepsilon}

\DeclareMathAlphabet{\mathup}{OT1}{\familydefault}{m}{n}
\newcommand{\dx}[1]{\mathop{}\!\mathup{d} #1}

\newcommand{\Leb}{\ensuremath{{L}}}

\DeclareMathOperator{\LSI}{LS}

\DeclarePairedDelimiter{\abs}{\lvert}{\rvert}
\DeclarePairedDelimiter{\norm}{\lVert}{\rVert}
\DeclarePairedDelimiter{\bra}{(}{)}
\DeclarePairedDelimiter{\pra}{[}{]}
\DeclarePairedDelimiter{\set}{\{}{\}}
\DeclarePairedDelimiter{\skp}{\langle}{\rangle}

\newenvironment{tenumerate}[1][]
  {\enumerate[label=\textup(\alph*\textup),ref=(\alph*),#1]}
  {\endenumerate}

\newcommand{\cE}{\ensuremath{\mathcal E}}
\newcommand{\cF}{\ensuremath{\mathcal F}}

\newcommand{\cI}{\ensuremath{\mathcal I}}

\newcommand{\cK}{\ensuremath{\mathcal K}}
\newcommand{\cL}{\ensuremath{\mathcal L}}

\newcommand{\cP}{\ensuremath{\mathcal P}}

\newcommand{\cX}{\ensuremath{\mathcal X}}
\newcommand{\cY}{\ensuremath{\mathcal Y}}

\newcommand{\C}{\ensuremath{\textrm{C}}}
\newcommand{\E}{\mathbb{E}}
\endlocaldefs

\usepackage{autonum}
\begin{document}
\begin{frontmatter}
 \title{Phase transitions, logarithmic Sobolev inequalities, and uniform-in-time propagation of chaos for weakly interacting diffusions} 

\runtitle{}

\begin{aug}

\author[A]{\fnms{Mat\'ias G.} \snm{Delgadino}\ead[label=e1]{matias.delgadino@math.utexas.edu}},
\author[B]{\fnms{Rishabh S.} \snm{Gvalani}\ead[label=e2]{gvalani@mis.mpg.de}},
\author[C]{\fnms{Grigorios A.} \snm{Pavliotis}\ead[label=e3]{g.pavliotis@imperial.ac.uk}}
\and
\author[D]{\fnms{Scott A.} \snm{Smith}\ead[label=e4]{ssmith@amss.ac.cn}}

\address[A]{Department of Mathematics, The University of Texas at Austin
\printead{e1}}

\address[B]{Max-Planck-Institut f\"ur Mathematik in den Naturwissenschaften
\printead{e2}}
\address[C]{Department of Mathematics, Imperial College London
\printead{e3}}
\address[D]{Academy of Mathematics and Systems Sciences, Chinese Academy of Sciences
\printead{e4}}
\end{aug}

\begin{abstract}

\end{abstract}



\end{frontmatter}

{\bf Abstract} In this article, we study the mean field limit of weakly interacting diffusions for confining and interaction potentials that are not necessarily convex. We explore the relationship between the large $N$ limit of the constant in the logarithmic Sobolev inequality (LSI) for the $N$-particle system and the presence or absence of phase transitions for the mean field limit. The non-degeneracy of the LSI constant is shown to have far reaching consequences, especially in the context of uniform-in-time propagation of chaos and the behaviour of equilibrium fluctuations. Our results extend previous results related to unbounded spin systems and recent results on propagation of chaos using novel coupling methods. As incidentals, we provide concise and, to our knowledge, new proofs of a generalised form of Talagrand's inequality and of quantitative propagation of chaos by employing techniques from the theory of gradient flows, specifically the Riemannian calculus on the space of probability measures.
\section{Introduction}

Interacting particle systems have attracted a lot of attention in recent years since they appear in diverse areas ranging from plasma physics and galactic dynamics to machine learning and optimization. For systems of identical (or exchangeable) particles in which the pair-wise interactions scale like the inverse of the number of particles, it is possible to pass to the mean field limit and obtain a coarse-grained description of the system via a nonlinear nonlocal PDE that governs the evolution of the one-particle density. In this paper, we consider systems of weakly interacting diffusions driven by pair-wise interactions, confinement and independent Brownian motions (see~\eqref{eq:particle system} ). In this case the mean field PDE is the so-called McKean--Vlasov equation. 

A natural problem that one would like to address is how to obtain sharp quantitative estimates on the rate at which the empirical measure of the particle system converges to the mean field limit, as the number of particles $N$ goes to infinity. When considering arbitrarily long time scales, this problem is intimately connected to the rate of convergence to steady states as time $t$ goes to infinity. For the study of such quantitative results, a crucial role is played by the Poincar\'{e} (PI)  and logarithmic Sobolev (LSI) inequalities. Our focus in this paper is to elucidate the connection between the validity of the LSI for the $N$-particle Gibbs measure uniformly in the number of particles $N$ and the properties of the mean field limit. We establish connections with uniform-in-time propagation of chaos, (non-)uniqueness of steady states of the mean field equation, exponential convergence to equilibrium, and the behaviour of equilibrium fluctuations. We show the effect that the presence of multiple steady states for the mean field equation, which correspond to invariant measures of the associated nonlinear McKean SDE, has on the various quantitative and qualitative features of the underlying interacting particle system. 

In the equilibrium statistical mechanics of lattice systems, the non-uniqueness of the infinite-volume grand canonical Gibbs measure is referred to as a phase transition~\cite{Georgii2011}. The presence of such a phase transition can then be detected with the help of some order parameter, for example the thermodynamic pressure for the nearest-neighbour Ising model, or the average magnetisation of the ensemble for the mean-field Curie--Weiss model. Based on the behaviour of these quantities (or rather on the behaviour of the infinite-volume partition function), one can then characterise the phase transitions as either continuous (second-order) or discontinuous (first-order).

The situation in our setting is more complicated but closely mirrors the one for lattice systems. Indeed, the system we consider can be thought of as spin system with mean field interaction and the ``spins'' taking values in some uncountable state space $\Omega$. The difference between the system we consider and the Ising and Curie--Weiss model lies in the fact that, except in a small number of specific examples, it is extremely hard to specify an order-parameter or understand the exact behaviour of the infinite-volume partition function.  An additional important difference is that for our system we will be more interested in the non-uniqueness of critical points of the free energy, which correspond to important changes in the features of the system, as opposed to non-uniqueness of its minimisers. In either case, the similarity with spin systems is instructive enough that it will serve the reader well to remember this analogy as we discuss the notion of phase transition we work with.

A detailed characterisation of phase transitions for McKean-Vlasov PDEs on the torus without a confining potential was given in~\cite{CGPS19}. In particular, the presence of phase transitions for this setting as it relates to non-uniqueness of minimisers of the free energy functional was discussed in detail. However, in this paper, we are more interested in characterising these phase transitions as they relate to non-uniqueness of critical points of the free energy.

For convex confining and interaction potentials (when the state space is Euclidean), the system does not undergo phase transitions. In fact, uniform-in-time propagation of chaos and uniqueness of the steady state for the mean field PDE have been established, for e.g. in~\cite{malrieu2003convergence}. Moreover, in \cite{CMV} under a uniform convexity assumption of the potentials, the authors show exponentially fast relaxation to the unique steady state of the mean field system. Our focus in this paper is to deal with non-convex potentials which may exhibit phase transitions and thus could not be expected to always (for all temperatures) exhibit uniform-in-time propagation of chaos. In terms of the LSI, we show that the existence of a non-minimising steady state implies the (quantitative) degeneracy of the constant in the LSI of the $N$-particle Gibbs measure in the limit as $N \rightarrow +\infty$. On the other hand, we show that the non-degeneracy of the LSI, implies uniform-in-time propagation of chaos and Gaussianity of the fluctuations around the mean field limit at equilibrium. Furthermore, we conjecture that the limit of the LSI constant for the $N$-particle system can be uniquely characterised in terms of the dissipation inequality for the mean field system.

The relation between the non-degeneracy of the constant in the PI or the LSI, the absence of phase transitions, and the exponentially fast decay of correlations has been studied extensively for unbounded spin systems~\cite{yoshida2003}. Conversely, in these works the equivalence between the slow decay of correlations and the fact that the constant in the LSI becomes degenerate at the phase transition has been established. Uniform estimates on the constant in the LSI beyond the convex case have been established recently~\cite{guillin2019uniform}, under the Lipschitzian spectral gap condition for the single particle. We remark that this assumption is reminiscent to the assumptions on the conditional measures for the two-scale LSI~\cite{OR07, GOVW09, Lel09}. We utilize the latter approach to show the non-degeneracy of the LSI for our $N$-particle Gibbs measure in the high temperature/weak interaction regime. 

In the probability literature, the study of the LSI in the context of linear Fokker--Planck equations goes back to the classical $\Gamma_2$ functional introduced by Bakry and Emery \cite{bakry2013analysis}. More recently, contractivity for interacting particle systems has been studied in the context of  entropic interpolation and Schr\"odinger bridges \cite{R,GLRT,BCGC,GLR}. These techniques yield proofs of both the Talagrand~\cite{CR} and Sobolev~\cite{dupaigne2020sobolev} inequality, under lower curvature conditions on the underlying manifold. We also mention the novel coupling techniques introduced by Eberle in~\cite{Ebe11} that produce contractivity estimates in a tailored transportation cost distance. This approach was then later used to prove uniform-in-time propagation of chaos estimates under a smallness assumptions on the interaction potential~\cite{DEGZ18}.

Our approach in this paper exploits the fact that both the $N$-particle system (or rather its Fokker--Planck equation) and its mean field limit are gradient flows of a particular energy functional with respect to the $2$-Wasserstein distance. We can use this structural feature of the system to study the limit of all the relevant quantities as $N\to\infty$. This approach was pioneered by Hauray and Mischler in \cite{hauray2014kac}, and later used by the authors in \cite{CARRILLO2020108734,delgadino2020diffusive} to study both propagation of chaos and periodic homogenization for the interacting particle system. The advantage of this approach is that we can often make minimal assumptions on the regularity of the confining and interaction potentials: we will essentially assume that they are both only semi-convex, which is natural for 2-Wasserstein gradient flows (see \cite{ambrosio2008gradient}). We refer to \cite{S,JW,JW2} for the reader interested in propagation of chaos results with more singular potentials.

\subsection*{Organization of the paper.} \cref{setup} sets up the problem we are interested in studying along with our notation and main assumptions. \cref{mainresults} contains the statements of all our main results. \cref{sec:phasetransition} connects our results to the phenomenon of phase transitions and discusses possible properties that could capture the radical change of behavior in our system in the presence multiple steady states of the mean field equation. \cref{sec:DFHS} contains some technical results on the convergence of the relevant quantities and functionals as $N\to\infty$ which play an important role in the proofs of our main theorems. \cref{sec:Talagrand,sce:degeneracy,sec:contraction,sec:uniquelmit,sec:uniform,sec:twoscale,sec:fluctuations}  contain the proofs of \Cref{thm:talagrand,thm:degeneracy,thm:contraction,thm:uniquelimit,thm:uniform,thm:twoscale,thm:fluctuations}, respectively.

\section{Set up, assumptions, and notation}\label{setup}
 We consider $\{X_t^i\}_{i=1,\dots,N}\subset \R^d$, the positions of $N$ indistinguishable interacting particles at time $t \geq 0$, satisfying the following system of SDEs:
\begin{equation}\label{eq:particle system}
\begin{cases}
\displaystyle
dX_t^i=-\nabla V(X_t^i)\dx{t}- \frac{1}{N}\sum_{j=1}^N\nabla_1 W(X_t^i,X_t^j) \dx{t}+\sqrt{2\beta^{-1}}dB_t^i\\
\displaystyle
\mathrm{Law}(X_0^1,\dots,X_0^N)=\rho_{\mathrm{\mathrm{in}}}^{\otimes N}\in\mathcal{P}_{2,\mathrm{\mathrm{sym}}}((\R^d)^N),
\end{cases}
\end{equation}
where $V:\R^d\to \R$, $W:\R^d\times \R^d\to \R$, $\beta^{-1}>0$ is the inverse temperature, $B_t^i,i=1,\dots,N$ are independent $d$-dimensional Brownian motions, and the initial position of the particles is i.i.d with law $\rho_{\mathrm{\mathrm{in}}}$. The chaoticity assumption on the initial data is not necessary but it greatly simplifies the exposition. Similarly, the state space $\R^d$ can be replaced by the periodic domain $\T^d$ or any convex set $\Omega\subset \R^d$ with normal reflecting boundary conditions, see, for example, \cite{sznitman1984nonlinear}. We denote the space of symmetric Borel probability measures on $\Omega^N$ with finite second moment by $\mathcal{P}_{2,\mathrm{\mathrm{sym}}}(\Omega^N)$, i.e. probability measures which are invariant under the relabeling of variables (or probability measures that arise as laws of exchangeable random variables). Throughout the paper, we will always work with probability measures that have finite second moment; to avoid burdensome notation, we forego the subscript $2$ from now on and simply write $\mathcal{P}_{\mathrm{\mathrm{sym}}}(\Omega^N)$. 

\noindent To ensure well-posedness of the evolutionary flows and coercivity, we make the following assumptions.
\begin{assumption}
    The confining potential $V$ is lower semicontinuous, bounded below, $K_V$-convex for some $K_V\in\R$ and there exists $R_0>0$ and $\delta>0$, such that $V(x)\ge |x|^\delta$ for $|x|>R_0$.
    \label{A1}
\end{assumption}
\begin{assumption}
    The interaction potential $W$ is lower semicontinuous, $K_W$-convex for some $K_W\in\R$, bounded below, symmetric $W(x,y)=W(y,x)$, vanishes along the diagonal $W(x,x)=0$, and there exists $C$ such that
\begin{equation}\label{doubling}
    |\nabla_1 W(x,y)|\le C(1+|W(x,y)|+V(x)+V(y))
\end{equation}
    \label{A2}
\end{assumption}
\begin{remark}
The $K$-Convexity assumptions on the potentials is short hand for global lower bounds on their Hessians
$$
D^2 V\ge K_V I^{d\times d}\qquad\mbox{and}\qquad D^2 W\ge K_W I^{2d\times 2d}
$$
with $K_V,$ $K_W\in \R$. Unlike a convexity assumption on the potentials, i.e. $K$-convexity with $K=0$ (see \cite{malrieu2003convergence} for results in the convex case), these assumptions are weak enough to include models that exhibit phase transitions with respect to changes in the system's temperature, $\beta^{-1}$, for example, the double well potential $V(x)=(1-|x|^2)^2$ with quadratic interactions $W(x,y)=|x-y|^2$, also known as the Desai--Zwanzig model (for more details see \cite{dawson1983critical}.
\end{remark}
\begin{remark}
The more technical bound \eqref{doubling} replaces the more classical doubling condition  \cite[Section 10.4.42]{ambrosio2008gradient} which is used to characterise the minimal sub-differential of the interaction energy \cite[Theorem 10.4.11]{ambrosio2008gradient}.
\end{remark}

\medskip

\subsection{The Fokker--Planck equation.} It is well known that the curve $\rho^N:[0,\infty)\to \mathcal{P}_{\mathrm{sym}}(\Omega^N)$ which describes the evolution of the law of the process $(X_t^1,...,X_t^N)\in \Omega^N$ satisfies the following linear Fokker--Planck equation
\begin{equation}\label{eq:forwardK}
\begin{cases}
\partial_t \rho^N =\beta^{-1}\Delta \rho^N+\nabla \cdot(\rho^N\;\nabla H_N)&\mbox{in $(0,\infty)\times \Omega^{N}$},\\
(\nabla \rho^N+\rho^N\nabla H_N)\cdot \vec{n}_{
\Omega^N}=0&\mbox{on $(0,\infty)\times \partial\Omega^{N}$}\\
\rho^N(0)=\rho_{\mathrm{in}}^{\otimes N},
\end{cases}
\end{equation}
where the Hamiltonian $H_N$ is given by
$$
H_N(x)=\sum_{i=1}^NV(x_i)+\frac{1}{2N}\sum_{j=1}^N\sum_{i=1}^N W(x_i,x_j),
$$
$\rho_{\mathrm{in}}\in \mathcal{P}(\Omega)$, and $\vec{n}_{\Omega^N}$ is the unit normal to $\Omega^{N}$.

\subsection{de Finetti/Hewitt--Savage.} To take the limit as $N\to \infty$, we will crucially use the exchangeability of the underlying particle system whose law is governed by~\eqref{eq:forwardK}. This implies that the joint law $\rho^N$ is symmetric for all times, that is
$$
\rho^N(t)\in \mathcal{P}_{\mathrm{sym}}(\Omega) \quad \textrm{ for all } t\geq0.
$$
The main idea is that we can characterise the limit $N\to\infty$ of $\mathcal{P}_{\mathrm{sym}}(\Omega^N)$ as $\mathcal{P}(\mathcal{P}(\Omega))$, which denotes the Borel probability measures with bounded second moment defined over the metric space $(\mathcal{P}(\Omega),d_2)$, where $d_2$ is the 2-Wasserstein distance.

Following de Finetti \cite{deFinetti} and Hewitt--Savage \cite{HewittSavage}, we know that any tight sequence $(\rho^N)_{N\in\N}$,  with $\rho^N\in\mathcal{P}_{\mathrm{sym}}(\Omega^{N})$, i.e. with tight (in $\cP(\Omega)$) $l$\textsuperscript{th} marginals for all $l \in \N$, has a limit  $P_\infty\in\mathcal{P}(\mathcal{P}(\Omega))$ along a subsequence which we do not relabel such that

$$
\rho^N\rightharpoonup P_\infty,
$$
where weak convergence is given by duality with cylindrical test functions. That is, for any $l\in\N$ and $\varphi\in C_c(\Omega^l)$, we have
\begin{equation}\label{def:DFHS}
    \lim_{N \to \infty}\int_{\Omega^l}\varphi(y)\;\dx \rho^N_l(y)=\int_{\mathcal{P}(\Omega)}\left(\int_{\Omega^{l}}\varphi(y)\;\dx\rho^{\otimes l}(y)\right)\;\dx P_\infty(\rho),
\end{equation}
where 
$$
\rho^N_l\in\mathcal{P}_{\mathrm{sym}}(\Omega^l)\qquad\mbox{is the $l$-th marginal of $\rho^N$.}
$$  
In essence, this means that in the limit $N\to\infty$  of symmetric probability measures can be characterised as convex combinations of chaotic measures. For more details, we refer the reader to~\cite{hauray2014kac,CARRILLO2020108734,rougerie2015finetti}. In the sequel, we will use the notation $\rho^N \rightharpoonup P_\infty \in \mathcal{P}(\mathcal{P}(\Omega))$ to denote this notion of weak convergence for any sequence $(\rho^N)_{N\in\N}$  with $\rho^N\in\mathcal{P}_{\mathrm{sym}}(\Omega^{N})$.

\noindent Moreover, a metric version of this result can be obtained by considering the appropriately scaled 2-Wasserstein distance, i.e. 
\begin{equation}\label{scaledd2}
    \overline{d}_2:=\frac{1}{\sqrt{N}}d_2,
\end{equation}
where $d_2$ is the classical 2-Wasserstein on $\mathcal{P}(\Omega^N)$. More specifically, Hauray--Mischler~\cite{hauray2014kac} showed that under the topology of~\eqref{def:DFHS}
\begin{equation}\label{frakd2}
    \lim_{N\to\infty} \overline{d}_2=\mathfrak{D}_2,
\end{equation}
where $\mathfrak{D}_2$ is the 2-Wasserstein distance defined on the probabilities with second moment bounded over the metric space $(\mathcal{P}_2(\R^d),d_2)$. For more details, see~\cref{thm:HM}. This metric on $\cP(\cP(\Omega))$ is closely related to the convergence discussed above as will be seen in~\cref{prop:compactness}.

\begin{figure}[h!]
\centering

\begin{tikzpicture}
\draw [black] plot [smooth, tension=1.5] coordinates {(-1,2) (0,0) (1,2)};
\draw[blue,fill=blue] (0,0) circle (.5ex);
\draw [black] plot [smooth, tension=1.1] coordinates {(2,2) (2.25,0.7) (3,0) (3.75,0.7) (4,2)};
\draw[blue,fill=blue] (3,0) circle (.5ex);
\draw [black] plot [smooth, tension=1.1] coordinates {(5,2) (5.25,-0.35) (6,0) (6.75,-0.35) (7,2)};
\draw[red,fill=red] (6,0) circle (.5ex);
\draw[blue,fill=blue] (5.43,-0.5) circle (.5ex);
\draw[blue,fill=blue] (6.57,-0.5) circle (.5ex);

\end{tikzpicture}

\begin{tikzpicture}
\node (A) at (0,0) {};
    \node (B) at (6,0) {};
    \draw[->] node[above] {$\beta<\beta_c$} (A) -- (B) node[above,midway] {$\beta=\beta_c$} node[above] {$\beta>\beta_c$}   ;
\end{tikzpicture}

\vspace{1em}
\begin{tikzpicture}
\draw [black] plot [smooth, tension=1.5] coordinates {(-1,2) (0,0) (1,2)};
\draw[blue,fill=blue] (0,0) circle (.5ex);
\draw [black] plot [smooth, tension=0.7] coordinates {(2,2) (2.25,0.7) (2.7,0.4) (3,0) (3.3,0.4) (3.75,0.7) (4,2)};
\draw[blue,fill=blue] (3,0) circle (.5ex);
\draw[red,fill=red] (2.5,0.47) circle (.5ex);
\draw[red,fill=red] (3.5,0.47) circle (.5ex);
\draw [black] plot [smooth, tension=1.1] coordinates {(5,2) (5.25,0.7) (5.7,0.6) (6,0) (6.3,0.6) (6.75,0.7) (7,2)};
\draw[blue,fill=blue] (6,0) circle (.5ex);
\draw[blue,fill=blue] (6.6,0.52) circle (.5ex);
\draw[red,fill=red] (6.4,0.62) circle (.5ex);
\draw[blue,fill=blue] (5.4,0.52) circle (.5ex);
\draw[red,fill=red] (5.6,0.62) circle (.5ex);
\end{tikzpicture}
\caption{A rough schematic showing two possible kinds of phase transition: The upper diagram shows a typical continuous phase transition. In this setting, the unique critical point (shown in blue) loses its local stability through a local (pitchfork) bifurcation which gives rise to new locally stable critical points. The lower diagram shows a typical discontinuous phase transition. In this setting, the unique critical point retains its local stability but new critical points arise in the free energy landscape through a saddle node bifurcation.}
\label{schematic}
\end{figure}
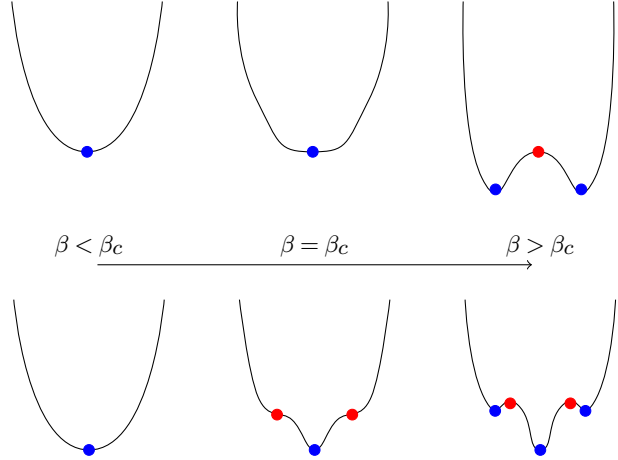

\subsection{The mean field limit and nonlinear behaviour} Within this formalism, the limit $N\to\infty$ of the equation \eqref{eq:forwardK} can be written as
\begin{equation}\label{propofchaos}
\rho^N(t)\rightharpoonup\delta_{\rho(t)}\in\mathcal{P}(\mathcal{P}(\Omega))\qquad\mbox{for all $t\ge 0$,}    
\end{equation}
where $\delta_{\rho(t)}$ denotes the delta measure concentrated on $\rho(t)$, the unique solution of the nonlinear McKean equation
\begin{equation}\label{McKean}
    \begin{cases}
    \partial_t\rho=\beta^{-1}\Delta\rho+\nabla\cdot\left(\rho\left(\nabla V+\nabla W\star \rho\right)\right)&\mbox{on $(0,\infty)\times \Omega$}\\
    (\nabla \rho+\rho(\nabla V+\nabla W\star \rho)\cdot\vec{n}_\Omega=0&\mbox{on $(0,\infty)\times \partial\Omega$}\\
    \rho(0)=\rho_{\mathrm{in}},
    \end{cases}
\end{equation}
with
\begin{equation}
    W\star \rho(x):=\int_{\R^d} W(x,y)\;\dx\rho(y) \, .
\end{equation}

\noindent One of the most salient differences between the particle dynamics \eqref{eq:forwardK} and the mean field dynamics \eqref{McKean} is that, whereas the Fokker-Planck equation governing the evolution of the $N$-particle system is linear, the mean field PDE~\eqref{McKean} is nonlinear. As is well known, a consequence of this is that for non-convex confining/interaction potentials the mean field dynamics might have more than one stationary state, in contrast to the particle dynamics. Indeed, the unique steady state of the $N$-particle Fokker--Planck equation ~\eqref{eq:forwardK} is given by the Gibbs measure
\begin{equation}\label{Gibbs}
M_N:=\frac{e^{-\beta H_N(x)}}{Z_N} \dx{x} \qquad\mbox{with}\qquad Z_N:=\int_{\Omega^N} e^{-\beta H_N(y)}\;\dx y.    
\end{equation}
On the other hand, the mean field limit can admit more than a single steady state, with the full characterisation being the set of solutions to the self-consistency equation
\begin{equation}\label{self-consistency}
    \beta^{-1}\log\rho_*+W\star\rho_*+V=C_*\qquad\mbox{on $\Omega$ for some $C_*\in\R$,}
\end{equation}
This is discussed in detail in~\cref{prop:steadystate}. See also~\cite{CGPS19, dawson1983critical} and the references therein.

\medskip

\subsection{Phase transitions}  The uniqueness/non-uniqueness of steady states of the mean field system~\eqref{McKean} depends both the temperature of the system (or, equivalently, the strength of the interaction) and on the convexity properties of the confining and interaction potentials $V$ and $W$. At sufficiently high temperatures, the diffusion is strong enough that the expected escape time of particles from local minima of the potentials is bounded uniformly in the number of particles. Indeed, a perturbation argument shows that, at sufficiently high temperatures, the self-consistency equation~\eqref{self-consistency} has a unique solution, see Proposition~\ref{prop:high_temp}. When we cool the system and the potentials are non-convex, particles can get trapped for arbitrarily long time scales in local minima and condense~\cite{BashiriMenz2021}. In statistical physics terms, the system changes from a gaseous state to a liquid or solid state. In the mean field limit, this change of behavior can be characterised by the local or global instability of the minimisers of the mean field energy (see~\eqref{meanfieldenergy}), see for instance \cite{CGPS19,chayes2010mckean} and \cref{schematic} for a loose picture of the change in the mean field energy landscape as the temperature of the system is varied. For more details on the possible definitions of a phase transition, see Section~\ref{sec:phasetransition}.

\medskip

\subsection{Qualitative long time behavior} Due to the linearity of the $N$-particle system, we know that \eqref{eq:forwardK} admits a unique steady state (given by the Gibbs measure $M_N$). If the potentials prevent mass from escaping to infinity, by the La-Salle's principle for graident flows \cite[Theorem 2.13]{carrillo2020invariance}, we know that independent of the initial condition
\begin{equation}\label{convergence1}
\lim _{t \to \infty }\rho^N(t)= M_N \, ,
\end{equation}
in the sense of weak convergence of probability measures. Furthermore, entropy methods and the LSI provide us with exponentially fast convergence to the steady state~\cite{Marko_Vill00}. On the other hand, if the McKean--Vlasov equation admits multiple steady states, then the limit $t\to\infty$ for the mean field dynamics will depend on the choice of initial condition. In this case, the limiting dynamics \eqref{McKean} do not approximate the particle dynamics \eqref{eq:particle system} for arbitrarily long times. In this paper we provide evidence that phase transitions constitute the natural obstruction to obtaining uniform-in-time propagation of chaos estimates. See Theorem~\ref{thm:uniform} for sufficient conditions for the mean field approximation to be valid uniformly in time.

\subsection{$\Gamma$-convergence} Taking advantage of the 2-Wasserstein gradient flow structure of \eqref{eq:forwardK} and \eqref{McKean} (see \cite{ambrosio2008gradient,CARRILLO2020108734, santambrogio2017euclidean}), the main tool that we will use to obtain a quantitative understanding of the limit $N\to\infty$ is $\Gamma$-convergence with respect to the topology introduced by de Finetti/Hewitt-Savage-type convergence \eqref{def:DFHS}. To illustrate this technique, we use it to characterise the limit $N\to\infty$ of the Gibbs measure $M_N$. Following the pioneering work of Messer and Spohn \cite{messer1982statistical}, we notice that $M_N$ is the unique minimiser over probability measures of the energy per unit particle
\begin{equation}\label{EnergyN}
E^N[\rho^N]=\frac{1}{N}\left(\beta^{-1}\int_{\Omega^N}\rho^N\log\rho^N\;\dx{x}+\int_{\Omega^N} H_N \rho^N\;\dx{x}\right).
\end{equation}
Taking $N\to\infty$, we can characterise the thermodynamic limit of the energy $E^N$ as $E^N\to^{\Gamma}E^\infty:\mathcal{P}(\mathcal{P}(\Omega))\to \R\cup\{+\infty\}$, where
\begin{align}
    E^\infty[P]:=\int_{\mathcal{P}(\Omega)}E^{MF}[\rho]\dx P(\rho),\label{Energyinfty}
\end{align}
with the mean field energy $E^{MF}:\mathcal{P}(\Omega)\to\R\cup\{+\infty\}$ given by 
\begin{equation}\label{meanfieldenergy}
E^{MF}[\rho]:=\beta^{-1}\int_{\Omega}\rho\log(\rho)\;\dx x+\frac{1}{2}\int_{\Omega^{2}}W(x,y)\;\dx \rho(x)\; \dx\rho(y)+\int_{\Omega}V(x)\; \dx \rho(x) \, ,
\end{equation}
see \cite{hauray2014kac} or \Cref{thm:MS} for a more modern proof. Using the fact that $M_N$ is the minimiser of \eqref{EnergyN}, we know that any accumulation point of the sequence $M_N$ (in the sense of de Finetti/Hewitt--Savage) $P_\infty\in\mathcal{P}(\mathcal{P}(\Omega))$ needs to be a minimiser of \eqref{Energyinfty}. Hence, we may conclude that the minimal energy converges
$$
\lim_{N\to\infty} \left(-\frac{1}{N}\log Z_N  \right)=\lim_{N} E^N[M_N]=E^\infty[P_\infty]=\inf_{P\in\cP(\cP(\R^d))}E^\infty[P] =\inf_{\rho\in\cP(\R^d)} E^{MF}[\rho],
$$
where we have used that \eqref{Energyinfty} is a potential energy, which also implies that $P_\infty$ needs to be supported on the minimisers of $E^{MF}$. This convergence and related results are discussed in further detail in~\cref{sec:DFHS}.

Under our previous hypothesis, we have the following standard result.
\begin{customthm}{A}\label{Ebound}
Under \cref{A1,A2}, $E^{MF}$ is bounded below and has at least one minimiser.
\end{customthm}
\noindent The lower bound follows from Jensen's inequality, while the existence of a minimiser follows from the direct method of calculus of variations. Under the extra assumption that the $E^{MF}$ admits a unique minimiser $\rho_\beta\in \mathcal{P}(\R^d)$, we have
$$
M_N\rightharpoonup \delta_{\rho_\beta}\in \mathcal{P}(\mathcal{P}(\Omega)) \, .
$$

\noindent The perspective of the proofs in this paper is that the evolution of the $N$-particle law \eqref{eq:forwardK} and the mean field limit~\eqref{McKean} are respectively the gradient flows of $E^N$ and $E^{MF}$ with respect to the scaled 2-Wasserstein distance $\bar d_2$ and the $2$-Wasserstein distance on $\cP(\Omega)$. In fact, the $\lambda$-convexity assumption on the potentials and the doubling condition \eqref{doubling} can be used to obtain uniqueness of the gradient flow solutions. The next fundamental result is essentially a restatement of~\cite[Theorem 11.2.8]{ambrosio2008gradient}, for the compact case see \cite{santambrogio2017euclidean}.

\begin{customthm}{B}\label{wellposedness}
If \cref{A1,A2} hold, then for any $\rho_{\mathrm{in}}\in\mathcal{P}(\Omega)$ there exists unique distributional solutions $\rho^N\in C([0,\infty);\mathcal{P}_{\mathrm{sym}}(\Omega^N)$ and $\rho\in C([0,\infty);\mathcal{P}(\Omega)$ to \eqref{eq:forwardK} and \eqref{McKean}, respectively, which are the gradient flows of $E^N$ \eqref{EnergyN} and $E^{MF}$ \eqref{meanfieldenergy} with respect to the scaled 2-Wasserstein distance $\bar d_2$ and the $2$-Wasserstein distance $d_2$ on $\cP(\Omega)$, respectively.
\end{customthm}

\noindent We remark here that in~\cite{CARRILLO2020108734} an alternative proof of propagation of chaos is provided which employs the $\Gamma$-convergence result and the convergence of the gradient flow structures.

\section{Main results}\label{mainresults}
 To quantify the convergence as $t\to\infty$ in \eqref{convergence1}, we can apply the standard relative entropy estimate~\cite{Varadhan1991}. More specifically, we consider the Lyapunov functional given by the scaled relative entropy of $\rho^N (t)$ with respect to the equilibrium measure $M_N$:
\begin{equation}\label{eq:Relentropy}
        E^N[\rho^N(t)]-E^N[M_N]=\overline{\mathcal{E}}(\rho^N(t)|M_N):=\frac{1}{N}\int_{\Omega^N}\log\left(\frac{\rho^N(t)}{M_N}\right)\rho^{N}(t)\;\dx x,
\end{equation}
where we use the notation $\overline{\mathcal{E}}:\mathcal{P}(\Omega)\times\mathcal{P}(\Omega)\to [0,\infty]$ to denote the scaled relative entropy. Taking a time derivative and using the PDE \eqref{eq:forwardK} we obtain the scaled relative Fisher information $\overline{\mathcal{I}}:\mathcal{P}(\Omega)\times\mathcal{P}(\Omega)\to[0,\infty]$:
\begin{align}\label{eq:dissipation}
   \frac{\dx}{\dx t}\overline{\mathcal{E}}(\rho^N(t)|M_N)
   =-\beta^{-1}\frac{1}{N} \int_{\Omega^N}\left|\nabla\log\left(\frac{\rho^N(t)}{M_N}\right)\right|^2\rho^{N}(t)\;\dx x
   =:-\beta^{-1} \overline{\mathcal{I}}(\rho^N(t)|M_N).
\end{align}
The convergence \eqref{convergence1} in relative entropy is exponential whenever we can show that the $N$-particle log Sobolev constant is bounded away from zero:
\begin{equation}\label{Nlogsobolevconstant}
   0<\lambda_{\LSI}^N:=\inf_{\rho^N\in\mathcal{P}(\Omega^N)\setminus \{M_N\}}\frac{\beta^{-1}\overline{\mathcal{I}}(\rho^N|M_N)}{\overline{\mathcal{E}}(\rho^N|M_N)}.
\end{equation}
Following the classical work of Bakry--Emery \cite{bakry2013analysis, Ledoux2001, Helffer2002}, we can find mild conditions for the positivity of the log Sobolev constant whenever the domain $\Omega$ is $\R^d$.
\begin{customthm}{C}\label{thm:BakryEmery}
Under ~\cref{A1,A2}, if there exists $R>1$ and $\lambda>0$ such that 
\begin{equation}\label{BEcondition}
D^2H^N(x)\ge \lambda I^{Nd\times Nd}\qquad\mbox{for every $|x|>R$,}
\end{equation}
then we have that $\lambda_{\LSI}^N>0$.
\end{customthm}
\begin{remark}
The convexity condition \eqref{BEcondition} in the far field can arise  from either the convexity of the interaction or the confining potential. We expect that the sharp condition for the Gibbs measure $M_N$ to satisfy $\lambda_{\LSI}^N>0$ uniformly in $N$ is related to the behavior of the mean field limit dissipation inequality \eqref{logsobolevconstant}. We will discuss this in more detail in \cref{conjecture}.
\end{remark}

For the mean field limit, we can perform a similar analysis with the relative mean field energy. More specifically, given $\rho(t)$ the solution to \eqref{McKean} we can differentiate to obtain the dissipation
\begin{equation}\label{dissipation}
    \frac{\dx}{\dx t} E^{MF}[\rho(t)]-\inf E^{MF}=-\int_{\Omega} |\beta^{-1}\nabla\log\rho(t)+\nabla W\star\rho(t)+\nabla V|^2\rho(t)\;\dx x=:-D(\rho(t)).
\end{equation}
Hence, we obtain exponential decay of the mean field energy to its minimum value, as long as the so-called infinite colume log Sobolev constant, given by
\begin{equation}\label{logsobolevconstant}
    0<\lambda_{\LSI}^\infty:=\inf_{\substack{\rho\in\mathcal{P}(\Omega)\\ \rho\notin \mathcal{K}}} \frac{D(\rho)}{E^{MF}[\rho(t)]-\inf E^{MF}},
\end{equation}
is positive, where
\begin{equation}
\mathcal{K}=\{\rho\in\mathcal{P}(\Omega)\;:\; E^{MF}[\rho]=\inf E^{MF}\}.    
\end{equation}
In both cases, when the log Sobolev constant is positive we can show that the relative energy behaves quadratically with respect to the 2-Wasserstein distance. This is essentially the content of Talagrand's inequality \cite{talagrand1995concentration}. One of our contributions in this paper is a new proof of a generalised version of this inequality using gradient flow techniques.
\begin{theorem}\label{thm:talagrand}
Under~\cref{A1,A2},
\begin{equation}\label{talagrand}
   E^N[\rho^N]-E^{N}[M_N]\ge  \frac{\lambda^N_{\LSI}}{2}\overline{d}^2_2(\rho^N,M_N)\qquad\mbox{and}\qquad E^{MF}[\rho]-\inf E^{MF}\ge  \frac{\lambda^\infty_{\LSI}}{2}d^2_2(\rho,\mathcal{K}),
\end{equation}
where $\mathcal{K}$ is the set of minimisers of $E^{MF}$ and $
d^2_2(\rho,\mathcal{K})=\inf_{\mu\in\mathcal{K}}d_2^2(\rho,\mu).$
\end{theorem}
\noindent An optimal transport-based proof of inequality \eqref{talagrand} for $E^N$ and $E^{MF}$ can be found in \cite[Theorem 22.17]{villani2008optimal}. We also refer to \cite{CR} for a proof using entropic interpolation. In \Cref{sec:Talagrand}, we provide a different more intuitive proof of \eqref{talagrand} for general energies $E:\mathcal{P}(\Omega)\to\R\cup \{+\infty\}$. Our strategy is to use the associated gradient flow structure and what is sometimes referred to as Otto calculus~\cite{Otto2001}, a formal Riemannian calculus on $(\mathcal{P}(\Omega),d_2)$. We should note that one of the main differences in Talagrand's inequality for the $N$-particle energy and for the mean field energy is that the set of minimisers $\mathcal{K}$ does not need to be a single point in the mean field case.

Having established the need for understanding the behavior of the log Sobolev constant, our first result relates the limit of the particle system log Sobolev constant \eqref{Nlogsobolevconstant} with the mean field or infinite volume log Sobolev constant\eqref{logsobolevconstant}:
\begin{theorem}\label{thm:degeneracy}
Under \cref{A1,A2}, we have
\begin{equation}
\limsup_{N\to\infty} \lambda_{\LSI}^N\le \lambda_{\LSI}^\infty.
\end{equation}
Moreover, if the mean field energy $E^{MF}$ \eqref{meanfieldenergy} admits a critical point that is not a minimiser, then $\lambda_{\LSI}^\infty=0$, and there exists $C>0$ such that
\begin{equation}
    \lambda_{\LSI}^N\le \frac{C}{N}.
\end{equation}
\end{theorem}

Our result complements similar results that have been obtained for unbounded spin systems~\cite{yoshida2003} and the references therein. On the other hand, when $\lambda^\infty_{\LSI}>0$, we can show that the {\it regularized log Sobolev constant}
\begin{equation}\label{e:LSI_reg}
    \lambda_{\LSI}^{N,\eps}:=\inf_{\rho^N:\;\overline{\mathcal{E}}(\rho^N|M_N)>\eps }\frac{\beta^{-1}\overline{\mathcal{I}}(\rho^N|M_N)}{\overline{\mathcal{E}}(\rho^N|M_N)}
\end{equation}
does not degenerate. More specifically,
\begin{equation}
    \lim_{N\to\infty}\lambda_{\LSI}^{N,\eps}\ge \lambda^\infty_{\LSI}>0.
\end{equation}
This result implies that relaxation to neighborhoods of the stationary state of the particle dynamics $\eqref{eq:forwardK}$ happens exponentially fast, uniformly in $N$. 
\begin{theorem}\label{thm:contraction}
Under \cref{A1,A2}, assume that $\lambda^\infty_{\LSI}>0$, and that $\rho_{\mathrm{\mathrm{in}}}$ in~\eqref{eq:particle system} has finite energy and bounded higher order moments, 
\begin{equation}\label{hyp:contraction}
E^{MF}[\rho_{\mathrm{\mathrm{in}}}]<\infty\qquad\mbox{and}\qquad \int_{\Omega}|x|^{2+\delta}\;\dx \rho_{\mathrm{\mathrm{in}}}<\infty,\qquad\mbox{for some $\delta>0$.}   
\end{equation}
Then, for every $\eps>0$, there exists $N_0 \in \N$, such that for every $N>N_0$ we have
$$
\overline{\mathcal{E}}(\rho^N(t)|M_N)\le \max\left\{\eps, e^{-\frac{1}{2}\lambda_{\LSI}^\infty t}\;\overline{\mathcal{E}}(\rho^{\otimes N}_{\mathrm{in}}|M_N)\right\}.
$$
\end{theorem}
\begin{remark}
For chaotic measures, we can take the limit $N\to\infty$ of the relative entropy obtaining
$$
\lim_{N \to \infty}\overline{\mathcal{E}}(\rho^{\otimes N}_{\mathrm{in}}|M_N)= E^{MF}[\rho_{\mathrm{in}}]-\inf E^{MF} \, .
$$
This is discussed in~\cref{thm:MS}.
\end{remark}

Unfortunately, we are not able to fully characterise the limit of $\lambda^N_{\LSI}$ in terms of the mean field limit. Despite this, Theorem~\ref{thm:degeneracy} and Theorem~\ref{thm:contraction}, provide us with evidence which is convincing enough to make the following conjecture.

\begin{conjecture}\label{conjecture}
Under \cref{A1,A2}, we have the equality
$$
\lim_{N\to\infty}\lambda^N_{\LSI}=\lambda^\infty_{\LSI}.
$$
\end{conjecture}

The results of our paper provide us with a strong indication that the absence of phase transitions (loosely defined to mean that the mean field limit has a unique stationary state), the non-degeneracy of the infinite volume log Sobolev constant, and the validity uniform-in-time propagation of chaos are all equivalent.

\subsection{Consequences of the non-degeneracy of the log Sobolev inequality} Bearing \cref{conjecture} in mind, we now explore the implications of the non-degeneracy of the LSI constant in the limit $N \rightarrow +\infty$. We begin by noticing that if the log Sobolev constant does not degenerate in $N$, then the invariant Gibbs measure $M_N$ of the $N$-particle system is well approximated by the unique minimiser of the mean field energy.
\begin{theorem}\label{thm:uniquelimit}
Under~\cref{A1,A2},  assume that $\limsup_{N\to\infty}\lambda^N_{\LSI}>0$. Then, there exists a unique steady state $\rho_\beta$ to \eqref{McKean}. Moreover, there exists $C>0$, such that
\begin{equation}
    \overline{d}^2_2(\rho_\beta^{\otimes N},M_N)\le \frac{2}{\lambda_{\LSI}^N}\overline{\mathcal{E}}(\rho_\beta^{\otimes N}|M_N)\le \frac{2}{(\lambda_{\LSI}^N)^2}\overline{\mathcal{I}}(\rho_\beta^{\otimes N}|M_N) \le \frac{C}{N}.
\end{equation}
\end{theorem}
Interpolating the previous result with more standard propagation of chaos estimates that depend on the convexity constant of the potentials, we obtain the following uniform-in-time propagation of chaos result.
\begin{theorem}\label{thm:uniform} 
Under \cref{A1,A2}, let $\rho^N$ and $\rho$ denote the unique solutions to the particle \eqref{eq:forwardK}
and mean field \eqref{McKean} dynamics given by \cref{wellposedness}. Assume that $\rho_{\mathrm{in}}$ has finite energy
$$
E^{MF}[\rho_{\mathrm{in}}]<\infty \, ,
$$
that the gradient of the square of the interaction potential is uniformly integrable
\begin{equation}\label{hypintegrability}
\sup_{t\in[0,\infty]}\int_\Omega |\nabla_1 W|^2\star \rho(t)\rho(t)\;\dx{x}<\infty\, ,    
\end{equation}
and that $\liminf_{N\to\infty}\lambda^N_{\LSI}=:\lambda^\infty>0$. Then,
\begin{equation}
    \overline{d}_2(\rho^N(t),\rho^{\otimes N}(t))\le \frac{C}{N^\theta}\qquad\mbox{for all $t>0$,}
\end{equation}
where
$$
\theta=\begin{cases}
1/2 &\mbox{if $K_V+K_W(1-1/N)>0$}\\
\frac{1}{2}\frac{\lambda^\infty}{\lambda^\infty-2(K_V+K_W(1-1/N))}&\mbox{if $K_V+K_W(1-1/N)< 0$}
\end{cases}
$$
with $K_V$ and $K_W$ the convexity constants of $V$ and $W$ in~\cref{A1,A2}. In the case where $K_V+K_W(1-1/N)=0$, we can pick any $\theta<1/2$.
\end{theorem}
\begin{remark}
The integrability assumption \eqref{hypintegrability} is trivially true when $W$ is uniformly Lipschitz. Also, this assumption is satisfied when the potentials are attractive enough in the far field, i.e. outside a ball of radius $R$, that we can obtain uniform exponential bounds for the tail behavior of the mean field solution.
\end{remark}

\begin{remark}
We note that for convex potentials $K_V,$ $K_W\ge 0$ the uniform propagation of chaos with $\theta=1/2$ has  already been shown in \cite{malrieu2003convergence}. The main difference with this work is that our approach utilizes the convexity of the entropy along the 2-Wasserstein distance to obtain a contraction estimate. This approach can be easily extended to manifolds with Ricci curvature bounded from below \cite{lott2009ricci}, where we need to consider the sign of
$$
K_{\mathrm{Ric}}+K_V+K_W(1-1/N)
$$
with $K_{\mathrm{Ric}}$ the lower bound on the Ricci curvature of the underlying Riemannian manifold.
\end{remark}

\begin{remark}
We do not expect $\theta$ in the above theorem to be sharp for $K_V+K_W(1-1/N)<0$. At sufficiently high temperatures and $K_V+K_W(1-1/N)<0$, a comparable result for the 1-Wasserstein distance is shown by coupling methods in \cite{Ebe16} with $\theta=1/2$. 
\end{remark}

We can also use the non-degeneracy of the LSI constant to identify the fluctuations at equilibrium.
\begin{theorem}\label{thm:fluctuations}
Let $\Omega=\T^d$ and assume that $\liminf_{N\to\infty}\lambda^N_{\LSI}>0$, then the fluctuations process
\begin{equation}
    \eta^N(t)=\sqrt{N}\left(\frac{1}{N}\sum_{i=1}^N\delta_{X_t^i}-\rho_\beta\right) \, ,
\end{equation}
where $(X_t^1,...,X_t^N)$ are solutions to~\eqref{eq:particle system} with initial law given by the invariant Gibbs measure $M_N$, satisfies
\begin{equation}
    \sup_{N \in \N, t\in [0,T]}\mathbb{E} \pra*{\|\eta^N(t)\|_{H^{-s}(\T^d)}^2}<\infty,
\end{equation}
for any $T>0,\,s>d/2+1$.

Moreover, assume that $V$, $W$ are smooth and that the linearised operator \eqref{McKean} around $\rho_\beta$
\begin{align}
\mathcal{L}_{\rho_\beta} \eta = \beta^{-1} \Delta \eta + \nabla\cdot (\rho_\beta \nabla W \star \eta) + \nabla\cdot (\eta \nabla W \star \rho_\beta) + \nabla \cdot (\nabla V \eta) \, ,
\end{align}
satisfies, for all $\phi \in \C^\infty(\T^d)$, the following coercivity inequality
\begin{equation}
\langle - \mathcal{L}_{\rho_\beta} \phi,\phi\rangle_{L^2(\T^d)} \ge c \|\nabla \phi\|^2_{L^2(\T^d)}
\label{eq:coercivity}
\end{equation}
for some $c>0$. 

Then, for any $m>d/2+3$, $\eta^N$ converges in law, as a $C([0,T];H^{-m}(\T^d))$-valued random variable, to the unique stationary solution $\eta^\infty$ of the following linear SPDE
\begin{equation}\label{SPDE}
    \partial_t\eta^\infty=\mathcal{L}_{\rho_\beta}\eta^\infty+\nabla\cdot(\sqrt{\rho_\beta}\xi)
\end{equation}
where $\xi$ is space-time white noise. 
\end{theorem}
\begin{remark}
We mention the result of Fernandez and Meleard \cite{fernandez1997hilbertian} (see also~\cite{TanakaHitsuda1981,ShigaTanaka1985}) which characterises the fluctuations of the particle dynamics with respect to the mean field limit for finite time horizons, under a stronger closeness assumption for the initial data. We also mention the recent preprint~\cite{wang2021gaussian} which studies fluctuations for singular potentials.

To our knowledge, the first available results for fluctuations at equilibrium is due to Dawson \cite{dawson1983critical} in which he shows that, for the specific case of Desai--Zwanzig model, at the phase transition temperature, equilibrium fluctuations are non-Gaussian. It is an interesting open problem if this behavior is universal for any system which undergoes a (continuous) phase transition. The different notions of phase transition will be discussed in~\cref{sec:phasetransition}.
\end{remark}

\begin{remark}
For The smoothness assumptions on $V$ and $W$ are used to have a well-defined semigroup associated to the linearised operator $\mathcal{L}_{\rho_\beta}$ which regularizes instantaneously arbitrary initial data in $H^{-m}(\T^d)$. Up to technical results, this can be quantified by requiring $V$ and $W\in W^{m+\epsilon,\infty}(\T^d)$ for any $\epsilon>0$. For fluctuation results with singular potentials we refer the reader to the recent preprint \cite{wang2021gaussian}.
\end{remark}

The unique invariant measure of the SPDE \eqref{SPDE} $\mathcal{G} \in \cP(H^{-m}(\T^d))$ is a centred Gaussian measure with covariance operator $Q_{\mathcal{G}}$ given by
\begin{equation}
    Q_{\mathcal{G}}(\varphi,\psi):=\lim_{t \to \infty }\int_0^t \int_{\T^d} \nabla e^{t\cL_{\rho_\beta}^*}\varphi \cdot \nabla e^{t\cL_{\rho_\beta}^*}\psi\, \rho_\beta\,\dx{x}  \dx{t} \, ,
\end{equation}
for any mean-zero $\varphi,\;\psi \in C^\infty(\T^d)$ and where $\cL_{\rho_\beta}^*$ denotes the flat $L^2$-adjoint.

In the specific case that $V \equiv 0$ and $W (x,y)=W(x-y)$, we can obtain a more explicit characterisation of $\mathcal{G}$. Indeed, since $\liminf_{N \to \infty} \lambda_{\LSI}^N >0$, we know from~\cref{prop:high_temp} (and the discussion following it) that $E^{MF}$ has a unique critical point which is given by $\rho_\infty(\dx{x})=\dx{x}$. We can then write down an explicit representation of the action of the semigroup $e^{t \mathcal{L}^*_{\dx{x}}}$ in Fourier space as follows
\begin{equation}
\hat{\varphi}_t(k) = e^{-4 \pi^2 \abs{k}^2(\beta^{-1}+ \hat{W}(k)) t} \hat{\varphi}(k) \quad k \in \Z^d, k \neq 0 \, ,
\end{equation}
where $\varphi_t = e^{t \mathcal{L}^*_{\dx{x}}} \varphi$ for some mean-zero $\varphi \in C^\infty(\T^d)$. This leaves us with the formula
\begin{equation}
Q_{\mathcal{G}}(\varphi,\psi) = \sum_{k \in \Z^d, k \neq 0}\frac{\hat{\varphi}(k) \hat{\psi}(k)}{8 \pi^2 (\beta^{-1} + \hat{W}(k))} \, ,
\end{equation}
where we have used the fact that the coercivity inequality~\eqref{eq:coercivity} is equivalent to the fact that $\beta^{-1} + \hat{W}(k) >0$ for all $k \in \Z^d, k \neq 0$, which is also equivalent to the condition $\beta<\beta_\sharp$, see \cref{prop1}.  

Another way of rewriting this is that $\mathcal{G}$ is the unique centred Gaussian measure with Cameron--Martin space given by the closure of all smooth mean-zero functions $\varphi$ under the norm
\begin{equation}
\norm{\varphi}_{\mathcal{H}_{\mathcal{G}}}^2 =  8 \pi^2 \bra*{\beta^{-1}\int_{\T^d}\varphi^2 \dx{x} + \int_{\T^d} (W* \varphi)\varphi \dx{x} }\, .
\end{equation}
Since, $\beta^{-1}+ \hat{W}(k)>0$ for all $k \in \Z^d, k \neq 0$, the above norm is equivalent to the standard $L^2(\T^d)$ norm. The above norm is also the same, up to a multiplicative constant, as the norm introduced in~\cref{cA}.

One can further use the structure of the covariance operator $Q_{\mathcal{G}}$ to read off that $\mathcal{G}$ is supported on $H^{-\frac{d}{2}-}(\T^d)$ distributions. Thus, the limiting equilibrium fluctuations have the regularity of spatial white noise, which is not surprising considering the fact that their Cameron--Martin space is ``\emph{basically}'' $L^2(\T^d)$.

\subsection{Non-degeneracy of the LSI constant in specific cases} Putting aside for the time being the validity of \cref{conjecture}, we show that the LSI constant $\lambda_{\LSI}^N$ does not degenerate in the high temperature regime when $\Omega$ is compact, or when the confinement $V$ satisfies an LSI inequality and the interaction strength is small enough.

\begin{theorem}
Assume that there exists a constant $C>0$ such that
    \begin{align}
     \norm{W}_{\Leb^\infty(\Omega^2)},\, \norm{D^2_{x y} W}_{\Leb^\infty(\Omega^2)} <C \,. 
    \end{align}
  We then have the following two scenarios:
  \begin{tenumerate}
  	\item \underline{Compact case}: Assume $\Omega$ is compact and its normalised Lebesgue measure $\dx{x}$ satisfies a log Sobolev inequality. Then, there exists a $0<\beta_{\LSI}=\beta_{\LSI}(C)$ such that for all $\beta<\beta_{\LSI}$, we have
  	\begin{equation}
  	\liminf_{N\to \infty} \lambda^N_{\LSI}>0 \, .
  	\end{equation} \label{twoscalea}
  	\item \underline{weak interaction case}: Assume $\Omega=\R^d$ and that the one-particle measure $Z_V^{-1}e^{-V} \dx{x}$ satisfies a log Sobolev inequality with constant $\lambda^V_{\LSI}>0$. Then there exists an $\epsilon_{\LSI}=\epsilon_{\LSI}(C, \lambda^V_{\LSI},\beta)>0$, such that for any $0\le \epsilon< \epsilon_{\LSI}$ we have
  	    \begin{equation}
  	        \liminf_{N\to \infty} \lambda^{\epsilon, N}_{\LSI}>0 \, ,
  	    \end{equation}
  	  where $\lambda^{\epsilon, N}_{\LSI}$ is the log Sobolev constant of the Gibbs measure $\tilde{M}_N= Z_N^{-1}e^{-\beta H^\epsilon_N} \dx{x}$ with
  	  \begin{equation}
      H^\epsilon_N(x)=\sum_{i=1}^NV(x_i)+\frac{\epsilon}{2N}\sum_{j=1}^N\sum_{i=1}^N W(x_i,x_j) \, .
  	\end{equation}\label{twoscaleb}
  \end{tenumerate} 
\label{thm:twoscale}
\end{theorem}

\noindent The above result relies crucially on the two-scale approach to log Sobolev inequalities introduced in~\cite{OR07}, which is based on the analysis of the marginal and conditional measures of $M_N$. For the convenience of the reader, we describe the main result of~\cite{OR07} in~\cref{thm:or}. In addition to the above high temperature result, it is also possible to obtain a sharper result in certain specific scenarios. For instance, in the case in which $\Omega=\T$, $W(x,y)=-\cos(2 \pi (x-y))$, and $V \equiv 0$. The corresponding system is referred to by many names: the noisy Kuramoto model, the mean field classical $XY$ model, the mean field $O(2)$ model or the Brownian mean field model~\cite{Chavanis2014}. It is known that this system exhibits a phase transition (of type A, B, and C, see~\cref{def:pt}). Due to the particularly simple nature of the model, it is possible to show that the $N$ particle log Sobolev constant $\lambda^N_{\LSI}$ is asymptotically non-degenerate all the way up to the critical inverse temperature $\beta_c=2$. We state without proof the following result due to Bauerschmidt and Bodineau~\cite{BB19}, see also~\cite{BeckerMenegaki2020}.
\begin{customthm}{D}[{\cite[Theorem 1]{BB19}}]
Consider the Gibbs measure $M_N$ of the mean field $O(2)$ model and denote by $\lambda_{\LSI}^N$ its log Sobolev constant. Then, for all $0 < \beta <\beta_c=2$, we have that
\begin{equation}
\liminf_{N\to \infty} \lambda^N_{\LSI}>0 \, .
\end{equation}
\label{thm:bb}
\end{customthm}
\begin{remark}
An essentially similar argument as in~\cite[Theorem 1]{BB19}, can be used to show that the system with $\Omega=\T$, $W(x,y)=-\cos(2 \pi k(x-y))$, and $V \equiv 0$ for some $k \in \N$ has a uniform LSI all the way up to the critical inverse temperature $\beta_c$ which coincides with $\beta_\sharp$ from \cref{prop1}.
\end{remark}

\section{Phase transitions}\label{sec:phasetransition}
We start our discussion by stating and proving the following result which provides us with a particularly useful characterisation of steady states.
\begin{proposition}\label{prop:steadystate}
Under \cref{A1,A2} the following statements are equivalent:
\begin{enumerate}
\item $\rho \in \cP(\Omega)$ is a critical point of the mean field free energy $E^{MF}$, that is to say 
$$
\abs{\partial E^{MF}}(\rho)=D(\rho)=0.
$$ 
\item $\rho \in \cP(\Omega)$ is a steady state of the mean field equation~\eqref{McKean}, i.e. it is distributional weak solution of the PDE
\[
\beta^{-1}\Delta \rho + \nabla \cdot(\rho (\nabla W \star \rho+\nabla V)) =0 \,  \quad x \in \Omega \, . 
\]
\item $\rho$ solves the self-consistency equation:
\begin{align}
\rho - \frac{1}{Z_\beta}e^{-\beta (W \star \rho +V) } =0 \, ,\qquad Z_\beta= \int_{\Omega} e^{-\beta (W \star \rho +V) } \dx{x} \, .
\label{selfconsistency}
\end{align}
\end{enumerate}
Furthermore, for all $\beta>0$, $E^{MF}$ has at least one global minimiser which is a critical point, and any critical point $\rho \in \cP(\Omega)$ of $E^{MF}$ is Lipschitz, strictly positive, and has moments of all orders.
\end{proposition}
\begin{proof}
The proof of the equivalence of the three characterisations follows from similar arguments to~\cite[Proposition 2.4]{CGPS19} and so we omit the proof. The fact that $E^{MF}$ has a critical point follows from \cref{Ebound}.

Now, using~\eqref{selfconsistency}, we know that any critical point $\rho \in \cP(\Omega)$ is of the form
\begin{align}
\frac{1}{Z_\beta}e^{-\beta (W \star \rho + V)} \, ,
\end{align}
which is Lipschitz and strictly positive. We now use~\cref{A2} to assert that, since $W$ is bounded below, 
\begin{align}
\rho \leq Z_\beta^{-1} e^{\beta C} e^{-\beta V} \, .
\end{align}
Thus, by~\cref{A1}, any critical point $\rho \in \cP(\Omega)$ of $E^{MF}$ has moments of all orders. 
\end{proof}
Furthermore, we have the following result regarding uniqueness of critical points. 
\begin{proposition}\label{prop:high_temp}
Under the assumptions of~\cref{thm:twoscale} there exists a unique critical point $\rho_\beta\in \cP(\Omega)$ of the mean field free energy $E^{MF}$.
\end{proposition} 
\begin{proof}
The proof of this result follows by combining the results of~\cref{thm:uniquelimit,thm:twoscale}. We know from~\cref{thm:twoscale} that the logarithmic Sobolev inequality holds uniformly for $M_N$. We can then use~\cref{thm:uniquelimit} to argue that for $\beta$ sufficiently small, the mean field free energy $E^{MF}$ must have a unique critical point $\rho_\beta$.
\end{proof}

When we discuss non-uniqueness of critical points, we will often restrict ourselves to the case in which $V=0, W (x,y)=W(x-y)$, and $\Omega=\T^d$. This case lends itself particularly well to analysis as the Lebesgue measure $\rho_\infty(\dx{x})= \dx{x}$, what we shall hereafter refer to as the ``\emph{flat state}'', is always a critical point of $E^{MF}$ for all $\beta>0$.  We shall see later that in this setting it is possible to provide relatively clean examples of the different types of phase transitions that we consider in this paper. For an example of what a typical phase transition looks like, see~\cref{schematic} for a schematic of the free energy landscape in the vicinity of a phase transition.

We now introduce and motivate a list of properties that may serve as a proxy for the absence of well-defined order parameter. We discuss how these properties relate to each other, specifically in the ``\emph{flat case}''. We start by looking at the linearisation of the mean field dynamics~\eqref{McKean}:

\begin{property}
\label{cA}
Fix $\beta>0$. We say that the system~\eqref{McKean} satisfies \emph{Property A} if $E^{MF}$ has a unique critical point $\rho_\beta$ and the linearised operator associated to right hand side of~\eqref{McKean} has a spectral gap under the interaction weighted inner product. More specifically, we have that
\begin{align}
\inf_{\eta \in C^\infty_0(\Omega),\int_\Omega \eta \dx{x}=0} \frac{\skp{- \mathcal{L}_{\rho_\beta} \eta ,\eta }_{W,\rho_\beta}}{\norm{\eta}_{W,\rho_\beta}^2} >0 \, ,
\tag{A}
\end{align}
where 
\begin{align}
\mathcal{L}_{\rho_\beta} \eta = \beta^{-1} \Delta \eta +\nabla \cdot(\nabla V \eta)+ \nabla\cdot (\rho_\beta \nabla W \star \eta) + \nabla\cdot (\eta \nabla W \star \rho_\beta) \, , 
\end{align}
and 
\begin{align}
\skp{\eta,\nu}_{W,\rho_\beta}= \frac{\beta^{-1}}{2} \int_{\Omega} \eta \nu \rho_\beta^{-1} \dx{x} + \frac{1}{2}\int_{\Omega}(W\star \eta)\nu \dx{x}
\end{align}
with $\norm{\eta}_{W,\rho_\beta}^2= \skp{\eta,\eta}_{W,\rho_\beta}$.
\end{property}
The above property captures essentially the local stability of the critical point $\rho_\beta$. Loss of the local stability of $\rho_\beta$ is a precursor to a phase transition. In the setting of a continuous phase transition (see the upper half of~\cref{schematic}), one expects this property to fail exactly at the critical temperature $\beta=\beta_c$. The weighted inner product arises naturally through the linearisation of the mean field log Sobolev constant, or more specifically through the linearisation of the mean field energy $E^{MF}$, see \eqref{Eexpansion}. We note that the bilinear form $\skp{\cdot,\cdot}_{W,\rho_\beta}$ is positive semi-definite when $\rho_\beta$ is the unique critical point. For the classical $XY$ model (see the discussion before~\cref{thm:bb}), it is known that the system exhibits a continuous phase transition~\cite[Proposition 6.1]{CGPS19}. In this case, the non-local inner product above it is related to the inner product that was introduced in \cite{BGP10} to understand the spectral gap ahead of the phase transition. In the absence of the interaction term, it reduces to the standard weighted inner product that symmetrises the Fokker-Planck operator, see~\cite{Pavliotis2014}. In the periodic translation invariant case $\Omega=\T^d$, $W(x,y)=W(x-y)$, and $V \equiv 0$, when $\rho_\infty=\dx{x}$ is the unique critical point of $E^{MF}$, the weighted inner product $\skp{\cdot,\cdot}_{W,\rho_\infty}$ is equivalent to the standard $L^2$ inner product. Thus, in this situation, checking~\cref{cA} is satisfied is equivalent to checking that $\mathcal{L}_{\rho_\infty}$ has a spectral gap in the standard $L^2$ inner product. This relationship will be made clearer in~\cref{prop1}. Before discussing~\cref{cA} in more detail, we introduce the next property which measures the local degeneracy of the self-consistency equation~\eqref{selfconsistency}.

\begin{property}
\label{cB}
Fix $\beta>0$. We denote by $T:\Leb^2(\Omega)\to\Leb^2(\Omega)$ the map
associated to the self-consistency equation~\eqref{selfconsistency}:
\begin{align}
T (\rho)= \rho - \frac{1}{Z_\beta}e^{-\beta (W \star \rho +V) } =0 \, ,\qquad Z_\beta= \int_{\Omega} e^{-\beta (W \star \rho +V) } \dx{x} \, .
\tag{B}
\end{align}
 We say that the system~\eqref{McKean} satisfies \emph{Property B} if $E^{MF}$ has a unique critical point $\rho_\beta$ and the Fr\'echet derivative $D_{\rho}T$ of $T$ at $\rho_\beta$ is non-degenerate, that is to say it has a trivial kernel consisting only of constant functions.
\end{property}
Under additional conditions, see \cite{CGPS19}, the above property being violated implies the presence of a local bifurcation around the minimiser $\rho_\beta$. A simple condition that implies the presence of a local bifurcation is when the algebraic multiplicity of the $0$ eigenvalue of $D_\rho T$ at $\rho_\beta$ is odd.  Local bifurcations also arise, if the map $T$ can be rewritten as a so-called potential operator and its Frech\'et derivative $D_\rho T$ has a non-zero crossing number, see ~\cite{Kie88}.  The above property exactly captures the second-order degeneracy of the mean field free energy around the unique critical point $\rho_\beta$. Indeed, formally expanding $E^{MF}[\rho]$ about $\rho_\beta$ we obtain
\begin{align}
E^{MF}[\rho]=& E^{MF}[\rho_\beta] + \int_{\Omega} \beta^{-1}\frac{\eta^2}{2}\rho_\beta^{-1}+\frac{1}{2} W \star \eta \eta\dx{x} + O(\eta^3)\, \\
=&E^{MF}[\rho_\beta] + \beta^{-1}\int_{\Omega} \bra*{D_{\rho}T[\rho_\beta]\eta}\eta\rho_\beta^{-1}\dx{x} + O(\eta^3) \, ,
\label{Eexpansion}
\end{align}
with $\eta=\rho-\rho_\beta$. Thus, if~\cref{cB} is satisfied the mean field free energy is non-degenerate at second order near the critical point $\rho_\beta$. 

The third and final property considers the validity of the Dissipation Inequality.
\begin{property}
\label{cC}
Fix $\beta>0$. We say that the system~\eqref{McKean} satisfies \emph{Property C} if the infinite-volume log Sobolev constant is positive. More precisely, if we have that
\begin{align}
\lambda^\infty_{\LSI}:=\inf_{\rho\in\mathcal{P}(\Omega), \rho \not \in \cK} \frac{D(\rho)}{E^{MF}[\rho]- \min_{\cP(\Omega)}E^{MF}}>0  \, .
\tag{C}
\label{eqC}
\end{align}
\end{property}
The third and final property captures global aspects of the free energy landscape. Indeed, one would expect it to be violated in both the situations described in~\cref{schematic}. For the case of the discontinuous phase transition, the lower half of~\cref{schematic},~\cref{cC} would be violated because of the presence of a non-minimising critical point of $E^{MF}$ which is represented by the red circles. Thus, the numerator of~\eqref{eqC} vanishes while the denominator is strictly positive. As an explicit example of a system which exhibits such a phase transition, one can consider $\Omega=\T$, with the bi-chromatic interaction potential $W(x,y)=-\cos(2 \pi(x-y))-\cos(4 \pi(x-y))$, and $V \equiv 0$ (cf.~\cite[Theorem 5.11]{CGPS19}). On the other hand, in the case of a continuous phase transition, the upper half of~\cref{schematic}, the fact that~\cref{cC} should be violated is more subtle. To observe this, we linearise the right hand side of~\eqref{eqC} about the unique minimiser $\rho_\beta$ at $\beta=\beta_c$. For the dissipation, we have that
\begin{align} 
D(\rho)=& 2 \int_{\Omega}\abs*{\beta^{-1} \nabla \frac{\eta}{\rho_\beta} +\nabla W \star \eta}^2 \dx{\rho_\beta} + O(\eta^3) \, ,
\end{align}
with $\eta=\rho-\rho_\beta$. Combining the above expression with~\eqref{Eexpansion}, we obtain to leading order
\begin{align}
\frac{D(\rho)}{E^{MF}[\rho]-E^{MF}[\rho_\beta]} \approx \frac{4 \int_{\Omega}\abs*{\beta^{-1} \nabla \frac{\eta}{\rho_\beta} +\nabla W \star \eta}^2 \dx{\rho_\beta}}{\int_{\Omega}\beta^{-1}\eta^2 + W\star \eta \eta \rho_\beta \dx{\rho_\beta^{-1}}} \, .
\label{linlogsob}
\end{align}

Moreover, we notice that
\begin{align}
(-\mathcal{L}_{\rho_\beta}\eta,\eta)_{W,\rho_\beta}=
& \frac{1}{2}\int_{\Omega}\abs*{\beta^{-1}\nabla \frac{\eta}{\rho_\beta}  + \nabla W \star \eta }^2 \dx{\rho_\beta} \, .
\end{align}
Using~\eqref{linlogsob}, it follows that to leading order
\begin{align}
\frac{D(\rho)}{E^{MF}[\rho]-E^{MF}[\rho_\beta]} \approx4\frac{\skp{-\mathcal{L}_{\rho_\beta}\eta,\eta}^2_{W,\rho_\beta}}{\norm{\eta}_{W,\rho_\beta}^2} \, .
\end{align}
Hence, in the setting of a continuous phase transition the infinite volume log Sobolev constant captures the spectral gap of $\mathcal{L}_{\rho_\beta}$, and thus also captures the loss of local stability of $\rho_\beta$.

 We now present the following result which characterises how the various properties we have discussed relate to each other in the periodic spatially homogeneous case. 
\begin{proposition}\label{prop1}
Assume $\Omega =\T^d$, $V\equiv 0,$ $W(x,y)=W(x-y)$, and denote
$$
\beta_\sharp:=\frac{1}{ \min(0,-\min_{k \in \Z^d \setminus \set{0}}\hat{W}(k))}\in(0,\infty].
$$
Then, for $\beta<\beta_\sharp$ the linearised operator $\mathcal{L}_{\dx{x}}$ satisfies the spectral gap property \eqref{A}, and the kernel of the Fr\'{e}chet derivative of the self-consistency equation \eqref{B} is non degenerate at the ``\emph{flat  state}''. For $\beta\geq \beta_\sharp $, \cref{cA,cB} are violated.

Furthermore, for all $\beta \neq \beta_\sharp$,~\cref{cC} implies~\cref{cA} and \cref{cB}. 
\end{proposition}
\begin{remark}
We emphasize that \cref{prop1} does not ensure that \cref{cA,cB} are satisfied for $\beta<\beta_\sharp$. In fact, for the bi-chromatic potential case $W(x,y)=-\cos(2 \pi(x-y))-\cos(4 \pi(x-y))$, we know that the ``\emph{flat  state}'' is not the unique steady state for some $\beta<\beta_\sharp$ violating the uniqueness requirement of \cref{cA,cB}, see~\cite[Theorem 5.11]{CGPS19}.
\end{remark}

\begin{proof}
We will first argue that~\cref{cA,cB} are equivalent.  It follows from the fact that $\rho_\infty(\dx{x}) \equiv \dx{x}$ is always a critical point of $E^{MF}$ (in the flat case) that if $E^{MF}$ has a unique critical point, it must be $\rho_\infty$.  For any mean-free $\eta \in C_0^\infty(\Omega)$, we have from the previous calculation, that
\begin{align}
\frac{\skp{-\mathcal{L}_{\rho_\infty} \eta,\eta}_{W,\rho_\beta} }{\norm{\eta}_{W,\rho_\beta}^2}=& \frac{1}{2}\frac{\int_{\Omega}\abs*{\beta^{-1}\nabla \eta  + \nabla W * \eta }^2\dx{x}}{ \int_{\Omega}(\beta^{-1}\eta + W * \eta )\eta \dx{x} } \, ,
\end{align}
where we have used the fact that $\rho_\beta$ is the Lebesgue measure on $\T^d$. It is easy to check that the above expression is strictly positive if and only if $\beta<\frac{1}{ \min(0,-\min_{k \in \Z^d \setminus \set{0}}\hat{W}(k))} =\beta_\sharp$. Similarly, computing the Fr\'echet derivative of $T$ at $\rho_\infty$, we obtain
\begin{align}
D_\rho T [\rho_\infty] \eta = \eta + \beta W * \eta  \, .
\end{align}
Again, the above linear operator has a trivial kernel if and only if $\beta < \beta_\sharp$. It follows that~\cref{cA,cB} are equivalent. 

We will now show that~\cref{cC} implies~\cref{cB} for $\beta \neq \beta_\sharp$. We know that~\cref{cB}  is satisfied if and only if $E^{MF}$ has unique critical point and $\beta<\beta_\sharp$. Assume it is violated. Then, either $E^{MF}$ has more than one critical point or $\beta \geq \beta_\sharp$ (or both). 

Consider the case in which $\beta<\beta_\sharp$ but $E^{MF}$ has more than one critical point one of which is $\rho_\infty$. Furthermore, we know from~\cite[Lemma 5.6]{GS20} that $\rho_\infty$ is a strict local minimum. If at least one of the  other critical points is not a minimiser then clearly~\cref{cC} is violated since the numerator can be chosen to be zero while the denominator remains positive. We will now argue that this is the case. Assume it is not, that is all other critical points are minimisers. Then, we can apply the mountain pass theorem in $\cP(\Omega)$  (using the fact that $\rho_\infty$ is a strict local minimum)~\cite[Theorem 1.1]{GS20} to construct a new non-minimising critical point thus obtaining a contradiction.

Consider now the case in which $\beta>\beta_\sharp$. In this situation, we note from~\cite[Proposition 5.3]{CGPS19} that $\rho_\infty$ is a non-minimising critical point of $E^{MF}$. Thus,~\cref{cC} is again violated.

\end{proof}
\begin{remark}
We note here that one would expect the~\cref{cA,cB,cC} to be equivalent to each other. However, we unable to show that this holds true in general.
\end{remark}
Given the above result, we are now finally in a position to present our definition of a phase transition in the ``\emph{flat  case}''.
\begin{definition}[Phase transition]
Assume $V\equiv 0, W(x,y)=W(x-y),$ and $\Omega =\T^d$. Then, we say that the system~\eqref{McKean} exhibits a phase transition of type A (resp. type B, type C) if there exists a $0<\beta_c<\infty$ such that for all $0<\beta<\beta_c$~\cref{cA} (resp. type B, type C) is satisfied, and for $\beta_c<\beta$~\cref{cA} (resp. type B, type C) is violated. 
\label{def:pt}
\end{definition}

Except for the Brownian mean field model see~\cref{thm:bb}, we cannot ensure that the existence of a phase transition in the sense of \cref{def:pt}. Combining \cref{thm:twoscale} with analysis in ~\cite{chayes2010mckean,CGPS19}, we can show the following result.

\begin{theorem}
Assume $V\equiv 0,\; W(x,y)=W(x-y),$ and $\Omega =\T^d$. If $W$ is $\mathbf{H}$-stable, that is to say $\hat{W}(k)\ge 0$ for all $k \in \Z^d \setminus \set{0}$, then the system~\eqref{McKean} satisfies \cref{cA,cB} for all $\beta\in(0,\infty)$.

If there exists $k \in \Z^d \setminus \set{0}$ such that $\hat{W}(k)<0$, then there exists $\beta_1<\beta_\sharp$ such that \cref{cA,cB,cC} are satisfied for all $\beta<\beta_1$, and \cref{cA,cB,cC} are violated for all $\beta>\beta_\sharp$. 
\end{theorem}

\begin{proof}
If $W$ is $\mathbf{H}$-stable, it follows from linear convexity $\rho_\infty=\dx{x}$ is the unique critical point of $E^{MF}$ for all $0<\beta<\infty$, hence \cref{cA,cB} cannot be violated, see \cite[Proposition 5.8]{CGPS19}. On the other hand, if $\beta_\sharp <\infty$, we know from \cref{prop1} that \cref{cA,cB,cC} are violated for all $\beta>\beta_\sharp$.

By \cref{thm:twoscale} and \cref{thm:degeneracy}, we know that there exists $\beta_{\LSI}$ such that $0<\beta<\beta_{\LSI}$ then \cref{cC} is satisfied. Finally, by \cref{prop1} we have that \cref{cA,cB} are also satisfied. 
\end{proof}

\section{The limit $\mathcal{P}_{\mathrm{sym}}(\Omega^N)\to\mathcal{P}(\mathcal{P}(\Omega))$}\label{sec:DFHS}
In this section, we recall some useful results that allows us the characterise the various relevant objects in the limit as $N\to \infty$. We follow the approach taken by Hauray and Mischler in~\cite{hauray2014kac}. We start by defining the empirical measure.
\begin{definition}\label{empirical}
Given some $\rho^N\in \mathcal{P}_{\mathrm{sym}}(\Omega^N)$, the empirical measure is the $\cP(\Omega)$-valued random variable which is given by
$$
\mu^{(N)}:=\frac{1}{N}\sum_{i=1}^N\delta_{x_i}\qquad\mbox{where $(x_1,...,x_N)$ is distributed according $\rho^N$.}
$$
We denote the law of the empirical measure $\mu^{(N)}$ on $\cP(\Omega)$ by
$$
\hat{\rho}^N\in\mathcal{P}(\mathcal{P}(\Omega)) \, .
$$
\end{definition}
The topology we consider throughout most of the manuscript is the one induced by the scaled 2-Wasserstein distance.
\begin{definition}\label{def:wasser_dist}
Given $\rho_1^N,\,\rho_2^N\in \mathcal{P}_{\mathrm{sym}}(\Omega^N)$, the scaled 2-Wassertein distance between them is given by
$$
\overline{d}_2(\rho_1^N,\rho_2^N)=\inf_{\substack{X\sim\rho_1^N\\Y\sim\rho_2^N}}\left(\frac{1}{N}\mathbb{E}[|X-Y|^2]\right)^{1/2},
$$
where $X,\, Y$ are $\Omega^N$-valued random variables.

Similarly, given $P_1,\,P_2\in \mathcal{P}(\cP(\Omega))$, the 2-Wasserstein distance between them is given by
$$
\mathfrak{D}_2(P_1,P_2)=\inf_{\substack{\mathcal{X}\sim P_1\\\mathcal{Y}\sim P_2}}\left(\mathbb{E}[d_2^2(\mathcal{X},\mathcal{Y})]\right)^{1/2}
$$
where $\cX,\, \cY$ are $\cP(\Omega)$-valued random variables.
\end{definition}
A fundamental result we need to understand the convergence $\mathcal{P}_{\mathrm{sym}}(\Omega^N)\to\mathcal{P}(\mathcal{P}(\Omega))$ is the fact that the mapping $\rho^N\mapsto\hat{\rho}^N$ is an isometry for the appropriately scaled 2-Wasserstein distance defined in ~\cref{def:wasser_dist}. Indeed, we have the following result.
\begin{theorem}[{~\cite[Proposition 2.14]{hauray2014kac}}]\label{thm:HM}
For each $\mu^N,\,\nu^N\in \mathcal{P}_{\sym}(\Omega^N)$, we have 
$$
\overline{d}_2(\mu^N,\nu^N)=\mathfrak{D}_2(\hat{\mu}^N,\hat{\nu}^N) \, .
$$
\end{theorem}
\begin{proof} 
We present only a sketch of the proof of this result. The first inequality $\overline{d}_2(\mu^N,\nu^N)\ge \mathfrak{D}_2(\hat{\mu}^N,\hat{\nu}^N)$ follows from the mapping in ~\cref{empirical}  $X\mapsto\mathcal{X}$, i.e. given $X=(X_1,...,X_N)$ an $\Omega^N$-valued random variable with law $\mu^N$, we define
$$
\cX:=\frac{1}{N}\sum_{i=1}^N\delta_{X_i} \, ,
$$
is a $\mathcal{P}(\Omega)$-valued random variable with law $\hat{\mu}^N$. The converse inequality follows from taking the inverse mapping and exploiting the symmetry of $\mu^N$ and $\nu^N$.
\end{proof}

Using this result, we can provide a metric notion of the de Finetti/Hewitt--Savage convergence.
\begin{definition}\label{def:metricDFHS}
We say that the sequence $(\rho^N)_{N\in\N}$ such that $\rho^N\in\cP_{\mathrm{\sym}}(\Omega^N)$ converges in the 2-Wasserstein distance to $P_\infty\in \cP(\cP(\Omega))$ if 
$$
\lim_{N\to\infty} \mathfrak{D}_2(\hat{\rho}^N,P_\infty)=0,
$$
where $\hat{\rho}^N$ is the law of the empirical measure (as $\cP(\Omega)$-valued random variable) associated to $\rho^N$.
\end{definition}
We now have the following result which connects the above convergence to the one introduced in~\eqref{def:DFHS}.

\begin{proposition}\label{prop:compactness}
For any sequence $(\rho^N)_{N\in\N}$ with $\rho^N \in \cP_{\mathrm{sym}}(\Omega^N)$ and

\begin{equation}\label{hyp:moment}
\sup_{N \in \N}  \int_{\Omega}|x|^\gamma\;\dx \rho_1^N(x)<\infty    
\end{equation}
for some $\gamma>2$, the metric convergence in Definition \ref{def:metricDFHS} is equivalent to the original marginal convergence of \eqref{def:DFHS}.
\end{proposition}

\begin{proof}
We first notice that by the results of Diaconis and Freedman~\cite{diaconis1980finite}, for a fixed $l\in\N$, the marginal $\rho^N_l$ coincides in the limit as $N \to \infty$ with the $l$\textsuperscript{th} product of the empirical measure. More specifically, for any $\varphi\in C^\infty_c(\Omega^l)$ we have that
\begin{equation}\label{eq:marginal}
\left|\int_{\Omega^l}\varphi \;\dx \rho^N_l-\int_{\cP(\Omega)}\left(\int_{\Omega^l}\varphi \;\dx \mu^{\otimes l}\right)\;\dx \hat{\rho}^N(\mu)\right|\le l^2\frac{\|\varphi\|_{L^\infty(\Omega^l)}}{N}  \, .
\end{equation}
   Furthermore,~\eqref{hyp:moment} implies that $\hat{\rho}^N$ is tight in $\cP(\Omega)$. Indeed, we have
 \begin{align}\label{2mom}
\int_{\cP(\Omega)}d_2^{2\gamma}(\rho,\delta_0)\;\dx \hat{\rho}^N(\rho)= \int_{\Omega}|x|^\gamma \; \dx \rho^N_1 \, .
\end{align} 
Now if $\hat{\rho}^N$ converges in the sense of~\cref{def:metricDFHS} to $P_\infty$ we know that it must converge when tested against every $C_b(\cP(\Omega))$. Since cylindrical test functions are a subset of $C_b(\Omega)$, it is clear from~\eqref{eq:marginal} that $\rho^N$ must also converge to $P_\infty$ in the sense of~\eqref{def:DFHS}.

On the other hand if $\rho^N$ converges to $P_\infty$ in the sense of~\eqref{def:DFHS} we know that, under~\eqref{hyp:moment}, the sequence $\hat{\rho}^N$ is relatively compact in $(\cP(\cP(\Omega)),\mathfrak{D}_2)$. By the Stone--Weierstrass theorem cylindrical functions are dense in $C_b(\cP(\Omega))$ (see \cite[Theorem 2.1 ]{rougerie2015finetti}), which tells us that the limits in~\eqref{eq:marginal} and ~\cref{def:metricDFHS} coincide.
\end{proof}
Next, we discuss the limit of the associated free energy and dissipation functionals.
\begin{theorem}\label{thm:MS}
Under \cref{A1,A2}, consider the sequence $(\rho^N)_{N\in\N}$ with $\rho^N\in\mathcal{P}_{\mathrm{sym}}(\Omega^N)$ such that $\rho^N\to P_\infty\in \cP(\cP(\Omega))$ in the sense of \Cref{def:metricDFHS}. Then,
\begin{equation}\label{liminf}
    \liminf_{N\to\infty} E^N[\rho^N]\ge E^\infty[P_\infty]\qquad\mbox{and}\qquad\liminf_{N\to\infty} \overline{\mathcal{I}}(\rho^N|M_N)\ge \int_{\mathcal{P}(\Omega)} D(\rho)\;dP_\infty(\rho)
\end{equation}
where $E^N$, $E^\infty$, $\overline{\mathcal{I}}$, and $D(\rho)$ are defined by \eqref{EnergyN}, \eqref{Energyinfty},~\eqref{eq:dissipation}, and \eqref{dissipation}, respectively. Moreover, for any $P_\infty\in\cP(\cP(\Omega))$ the sequence
$$
\rho^N_*=\int_{\mathcal{P}(\Omega)}\rho^{\otimes N}\;\dx P_\infty(\rho)\in \mathcal{P}_{\mathrm{sym}}(\Omega^N)
$$
attains the limits the limits, i.e.
\begin{equation}\label{lim}
    \lim_{N\to\infty} E^N[\rho_*^N]= E^\infty[P_\infty]\qquad\mbox{and}\qquad \lim_{N\to\infty} \overline{\mathcal{I}}(\rho_*^N|M_N)=\int_{\mathcal{P}(\Omega)} D(\rho)\;dP_\infty(\rho)
\end{equation}
\end{theorem}
\begin{proof}
The convergence of the relative entropy functional follows directly from the classical arguments in \cite{messer1982statistical}. A more modern proof can be found in~\cite{hauray2014kac}. The convergence of the relative Fisher information with respect to the Lebesgue measure is covered in \cite{hauray2014kac}. Our case is slightly more involved due to the minimal regularity assumptions on the potentials. Formally, expanding the square we obtain after integrating by parts
\begin{align}
\overline{\mathcal{I}}(\rho^N,M_N)=&\frac{1}{N}\int_{\Omega^N} |\nabla \log \rho^N-\nabla H^N|^2 \dx \rho^N\\=&\int_{\Omega^N} (|\nabla \log \rho^N|^2+2\Delta H^N +|\nabla H^N|^2)\dx \rho^N \, . \label{regularexp}
\end{align}
The first term of the above expression is exactly the relative Fisher information with respect to the Lebesgue measure and so is already covered in~\cite{hauray2014kac}. Under stronger regularity assumptions on the potentials, the second and third term fall within the type of functionals already considered in the classical~\cite{messer1982statistical}. The main obstruction to conclude under \cref{A1,A2} is that $\Delta V$ and $\Delta W$ are merely signed measures bounded below, so we need to adapt the proofs of \cite{hauray2014kac} to a lower regularity setting. We will circumvent the regularity problem, by appealing again and again to convexity. More specifically, by \cref{A1,A2}, we know that
\begin{equation}\label{convexity}
\overline{\mathcal{I}}^{1/2}(\rho^N | M_N)=\sup_{\nu^N\in\mathcal{P}_{\mathrm{sym}}(\Omega^N)} \left(\frac{E^N[\rho^N]-E^N[\nu^N]}{\overline{d}_2(\rho^N,\nu^N)}-\frac{K_V+K_W}{2}\overline{d}_2(\rho^N,\nu^N)\right)_+   \, , 
\end{equation}
see \cite[Theorem 2.4.9]{ambrosio2008gradient}. We pick $\nu^N$ to be the recovery sequence for a generic $Q_\infty\in\cP(\cP(\Omega))$, using the lower semicontinuous convergence of $E^N$, and the isometry \cref{thm:HM}, to obtain
$$
\liminf_{N\to\infty} \overline{\mathcal{I}}^{1/2}(\rho^N | M_N)\ge \sup_{Q_\infty\in\cP(\cP(\Omega))}\left(\frac{E^\infty[P_\infty]-E^\infty[Q_\infty]}{\mathfrak{D}_2(P_\infty,Q_\infty)}-\frac{K_V+K_W}{2}\mathfrak{D}_2(P_\infty,Q_\infty)\right)_+.
$$
To obtain equality, we consider $S_t:\cP(\Omega)\to\cP(\Omega)$ the solution operator associated to \eqref{McKean} which is well-defined by \cref{wellposedness} and consider the curve $Q_\infty=S_t\#P_\infty$, which coincides with the unique gradient flow of $E^\infty$ with respect to $\mathfrak{D}_2$ (see \cite[Lemma 19]{CARRILLO2020108734}). Taking $t\to 0^+$, we obtain
\begin{align}
\lim_{t\to 0^+}&\left(\frac{E^\infty[P_\infty]-E^\infty[S_t\#P_\infty]}{\mathfrak{D}_2(P_\infty,S_t\#P_\infty)}-\frac{K_V+K_W}{2}\mathfrak{D}_2(P_\infty,S_t\#P_\infty)\right)_+\\=&\left(\int_{\mathcal{P}(\Omega)} D(\rho)\;\dx P_\infty(\rho)\right)^{1/2} \, ,
\end{align}
obtaining the desired lower semicontinuity result.
\begin{equation}\label{eq:lim1}
\liminf_{N\to\infty}\mathcal{I}(\rho^N|M_N)\ge \int_{\mathcal{P}(\Omega)} D(\rho)\;\dx P_\infty(\rho)    
\end{equation}

Now we focus on showing that given $P\in\cP(\cP(\Omega))$, the sequence
\begin{equation}\label{recovery}
\rho_*^N=\int_{\cP(\Omega)}\rho^{\otimes N} \;\dx P(\rho)    
\end{equation}
attains the limit~\eqref{lim}. For convenience, we consider the auxiliary functional $\mathcal{I}':\cP(\cP(\Omega))\to \R\cup\{\infty\}$, which is given by
$$
\mathcal{I}'[P]=\lim_{N\to\infty}\overline{\mathcal{I}}(\rho_*^N|M_N),
$$
where $\rho_*^N$ is given by~\eqref{recovery}. Next, we show that $\mathcal{I}'$ is well defined and that it is weakly lower semicontinuous.

To show this, we regularize $\rho_*^N$ by using the $N$-particle dynamics~\eqref{eq:forwardK}. Using the convexity in \cref{A1,A2}, we obtain that 
$$
e^{t(\lambda_V+\lambda_W)}\overline{\mathcal{I}}(S^N_t\rho_*^N|M_N)\le\overline{\mathcal{I}}(\rho_*^N|M_N) \, ,
$$
where $S^N_t:\cP(\Omega^N)\to\cP(\Omega^N)$ is the solution operator associated to the Fokker--Planck equation \eqref{eq:forwardK}. Using the weak lower semicontinuity of $\overline{\mathcal{I}}$ from the expression \eqref{convexity}, we obtain
\begin{equation}\label{limit}
\lim_{t\to 0^+}\overline{\mathcal{I}}(S^N_t\rho_*^N|M_N)=\overline{\mathcal{I}}(\rho_*^N|M_N) \, .  
\end{equation}
Using the fact that $S_t^N$ is immediately regularizing, we can use the formal expression \eqref{regularexp} freely for the measure $S^N_t\rho_*^N$. By using the sub-additivity of the Fisher information, we can notice that for any $3<j<N$
\begin{equation}\label{eq:finiteN}
\overline{\mathcal{I}}((S^N_t\rho_*^N)_j|M_j)\le \overline{\mathcal{I}}(S^N_t\rho_*^N|M_N) \, ,
\end{equation}
where $(S^N_t\rho_*^N)_j$ is the $j$\textsuperscript{th} marginal of $S^N_t\rho_*^N$. Taking $t\to 0^+$, using \eqref{limit} and the weak lower semicontinuity from~\eqref{convexity}, we obtain for any $3<j<N$
$$
\overline{\mathcal{I}}(\rho^j_*|M_j)\le \overline{\mathcal{I}}(\rho_*^N|M_N) \, ,
$$
which implies that we can characterise
$$
\mathcal{I}'[P]=\sup_{N\ge 3}\overline{\mathcal{I}}(\rho_*^N|M_N).
$$
By the weak lower semicontinuity of $\overline{\mathcal{I}}$, we obtain the weak lower semicontinuity of $\mathcal{I}'$.

Using \eqref{eq:finiteN} and the propagation of chaos result of \cite{CARRILLO2020108734}, we obtain that
\begin{equation}\label{eq:212}
e^{t(\lambda_V+\lambda_W)}\overline{\mathcal{I}}(\rho_{t,*}^j|M_j) \le e^{t(\lambda_V+\lambda_W)}\liminf_{N\to\infty}\overline{\mathcal{I}}((S^N_t\rho_*^N)_j|M_j)\le \mathcal{I}'[P] \, ,    
\end{equation}
where
$$
\rho_{t,*}^j=\int_{\cP(\Omega)}\rho^{\otimes j}\;\dx (S_t\#P)(\rho)=\int_{\cP(\Omega)}(S_t\rho)^{\otimes j}\;\dx P(\rho)
$$
with $S_t$ the solution operator to~\eqref{McKean}. Taking $j\to\infty$ in \eqref{eq:212}, we obtain the inequality
$$
e^{t(\lambda_V+\lambda_W)}\int_{\cP(\Omega)} D(S_t\#\rho)\;\dx P(\rho) =e^{t(\lambda_V+\lambda_W)}\mathcal{I}'[S_t\#P]\le  \mathcal{I}'[P] \, ,
$$
where the equality in the left hand side follows from using the expression \eqref{regularexp} and the result for the standard Fisher information in \cite{hauray2014kac}. Taking $t\to0^+$ and using the lower semicontinuity of $\mathcal{I}'$ we have showed above, we obtain the desired conclusion \eqref{lim}
$$
\mathcal{I}'[P]=\int_{\cP(\Omega)} D(\rho)\;\dx P(\rho).
$$
\end{proof}

Standard arguments also yield the convergence of the minimal energy.
\begin{corollary}
Any accumulation point $P_\infty\in\cP(\cP(\Omega))$ of the sequence $(M_N)_{N\in\N}$ satisfies 
$$
\mathrm{supp}(P_\infty)\subset \{\rho\;:\; E^{MF}[\rho]=\min_{\cP(\Omega)} E^{MF}\}
$$
and
\begin{equation}\label{e:part_fn_converg}
E^N[M_N]=-\frac{1}{N}\log(Z_N)\stackrel{N \to \infty}{\to} \min_{\cP(\Omega)} E^{MF} \, .
\end{equation}
\end{corollary}

We remark that the convergence~\eqref{e:part_fn_converg} coincides with the standard definition of the thermodynamic limit from statistical mechanics (see~\cite[Ch. 3]{Helffer2002}).

\begin{theorem}[HWI inequality]\label{thm:HWI}
Let $V$ and $W$ be $K$-convex.  Given two arbitrary probability measures $\mu^N,\,\nu^N\in\cP_{\mathrm{sym}}(\Omega^N)$, it holds
\begin{equation}\label{HWI}
    \overline{\mathcal{E}}(\mu^N|M_N)\le  \overline{\mathcal{E}}(\nu^N|M_N)+\overline{d}_2(\mu^N,\nu^N)\sqrt{\overline{\mathcal{I}}(\mu^N|M_N)}-\frac{K}{2}\overline{d}^2_2(\mu^N,\nu^N) \, .
\end{equation}
\end{theorem}

\begin{proof}
The proof of this result can be found in \cite[Theorem 30.22]{villani2008optimal}. We just present the main idea of the proof for the reader's convenience. We consider $\gamma:[0,1]\to\cP_{\mathrm{sym}}(\Omega^N)$ the 2-Wasserstein geodesic between $\mu^N$ and $\nu^N$. Under the hypothesis of $K$ convexity of the potentials, we obtain that $t\to \overline{\mathcal{E}}(\gamma(t)|M_N)$ is $K$-convex, which implies the desired inequality~\eqref{HWI}.
\end{proof}
\begin{corollary}\label{cor:convergence}
Assume that~\cref{A1,A2} hold true and that the sequence $(\rho^N)_{N\in\N}$ with $\rho^N \in \cP_{\sym}(\Omega^N)$ such that $\rho^N$ converges to $P_\infty \in \cP(\cP(\Omega))$ in the sense of~\cref{def:metricDFHS} and
\begin{equation}\label{hyp:boundentropy}
    \sup_{N \in \N} \overline{\mathcal{I}}(\rho^N|M_N)<\infty \, ,
\end{equation}
then
\begin{equation}
    \lim_{N\to\infty}\overline{\mathcal{E}}(\rho^N|M_N)=E^\infty[P_\infty]-\min_{\cP(\Omega)} E^{MF}.
\end{equation}
\end{corollary}
\begin{proof}
Given $P_\infty$, we consider the recovery sequence $(\rho^N_*)_{N \in \N}$ from~\cref{thm:MS}. By the HWI inequality~\eqref{HWI}, we have
\begin{equation}
    \overline{\mathcal{E}}(\rho^N|M_N)\le  \overline{\mathcal{E}}(\rho_*^N|M_N)+\overline{d}_2(\rho^N,\rho_*^N)\sqrt{\overline{\mathcal{I}}(\rho^N|M_N)}-\frac{K}{2}\overline{d}^2_2(\rho^N,\rho_*^N) \, .
\end{equation}
Using that both sequences converge to $P_\infty$ and that $(\rho^N_*)_{N \in \N}$ is a recovery sequence, we obtain
\begin{equation}
\limsup_{N\to\infty} \left(E^N[\rho^N]-E^N[M_N]\right)    =\limsup_{N\to\infty} \overline{\mathcal{E}}(\rho^N|M_N)
\le E^\infty[P_\infty]-\min_{\cP(\Omega)} E^{MF}\, .
\end{equation}
The reverse inequality follows from~\cref{thm:MS}.
\end{proof}
\begin{remark}
A version of this result without the potentials $V$ and $W$ can be found in~\cite{hauray2014kac}.
\end{remark}

\section{Proof of~\Cref{thm:talagrand}}\label{sec:Talagrand}
\begin{proof}[Proof of \cref{thm:talagrand}]
We prove the result only for $E^{MF}$, as the proof for $E^N$ and even more general energies $E:\mathcal{P}(\Omega)\to[0,\infty]$ is analogous, see \cref{rem:gE}. We consider $\rho:[0,\infty)\to \mathcal{P}(\Omega)$, the unique 2-Wasserstein gradient flow of $E^{MF}$  with initial condition $\rho_0 \in \cP(\Omega)$, see \cref{wellposedness}. 

We notice that if $\lambda^\infty_{\LSI}=0$, then there is nothing left to prove. Hence, we can assume without loss of generality that $\lambda^\infty_{\LSI}>0$, which implies (cf.~\cref{cC,prop:steadystate}) that $E^{MF}$ does not admit non-minimising steady state. Using the version of LaSalle's invariance principle for gradient flows proved in~\cite[Theorem 4.11]{carrillo2020invariance}, we know that $\rho(t)$ accumulates on the set of steady states of $E^{MF}$, as $t\to\infty$. Using the fact that all steady states are minimisers, we can find a sequence of times $t_n\to\infty$ such that 
\begin{equation}\label{eq:111}
    \rho(t_n)\rightharpoonup\rho_\infty
\end{equation}
for some $\rho_\infty \in \mathcal{K}$. Differentiating $E^{MF}[\rho(t)]$ with respect to time and using that $\lambda^\infty_{\LSI}>0$ we obtain
\begin{align}
\frac{\dx{}}{\dx{t}}\bra*{E^{MF}[\rho(t)]-\inf_{\cP(\Omega)} E^{MF}}&=-\int_{\Omega}\abs*{\beta^{-1}\nabla \log \rho(t)+\nabla V+\nabla W\star\rho(t)}^2\;\dx{\rho(t)}\\
\le& -\sqrt{\lambda_{\LSI}^\infty} \bra*{E^{MF}[\rho(t)]-\inf_{\cP(\Omega)} E^{MF}}^{1/2}\\
&\times\left(\int_{\Omega}\abs*{\beta^{-1}\nabla \log \rho(t)+\nabla V+\nabla W\star\rho(t)}^2\;\dx{\rho}(t)\right)^{1/2}.
\end{align}
Integrating from $0$ to $t_n$ for some $n \in \N$, we obtain
\begin{align}
&\bra*{E^{MF}[\rho(t_n)]-\inf_{\cP(\Omega)} E^{MF}}^{1/2}-\bra*{E^{MF}[\rho_0]-\inf_{\cP(\Omega)} E^{MF}}^{1/2}\\
\le& -\left(\frac{\lambda_{\LSI}^\infty}{2}\right)^{1/2}\int_0^{t_n}\left(\frac{1}{2}\int_{\Omega}|\beta^{-1}\nabla \log \rho(s)+\nabla V+\nabla W\star\rho(s)|^2\;\dx{\rho}(s)\right)^{1/2}\;\dx{s}\\
\le& -\left(\frac{\lambda_{\LSI}^\infty}{2}\right)^{1/2} d_2(\rho_0,\rho(t_n)),
\end{align}
where the last inequality follows from the Benamou--Brenier~\cite{benamou2000computational} formulation of the 2-Wasserstein distance. Taking $t_n\to\infty$, applying \eqref{eq:111}, and rearranging, we end up with
\begin{align}
\bra*{E^{MF}[\rho_0]-\inf_{\cP(\Omega)}E^{MF}}^{1/2}\ge \left(\frac{\lambda_{\LSI}^\infty}{2}\right)^{1/2} d_2(\rho_0,\rho_\infty)\ge \left(\frac{\lambda_{\LSI}^\infty}{2}\right)^{1/2} d_2(\rho_0,\mathcal{K}). 
\end{align}
The desired inequality \eqref{talagrand} follows by squaring both sides.
\end{proof}
\begin{remark}\label{rem:gE}
This proof can easily be adapted to general $E:\mathcal{P}(\Omega)\to(0,\infty]$, that are regular enough to admit gradient flow solutions from arbitrary initial data, and that sub-level sets are weakly compact to ensure convergence of the gradient flows to steady states \cite[Theorem 2.12]{carrillo2020invariance}.
\end{remark}

\section{Proof of Theorem~\ref{thm:degeneracy}}\label{sce:degeneracy}
\begin{proof}[Proof of Theorem~\ref{thm:degeneracy}]
For the proof of the first part of the theorem, we consider some $\rho\in\mathcal{P}(\Omega)$ such that $E^{MF}[\rho]>\inf_{\cP(\Omega)} E^{MF}>-\infty$, see \cref{Ebound}. By \cref{thm:MS} and~\cref{cor:convergence}, we can find a recovery sequence $P^N=\rho^{\otimes N}$, such that
\begin{equation}
\lim_{N \to \infty}\overline{\mathcal{E}}(P^N|M_N)= E^{MF}[\rho]-\inf_{\cP(\Omega)} E^{MF} \qquad\mbox{and}\qquad \lim_{N \to \infty} \overline{\mathcal{I}}(P^N|M_N)= D(\rho).
\end{equation}
Using the definition of the log Sobolev constant and passing to the limit as $N \to \infty$, we observe
$$
\limsup_{N\to\infty}\lambda^N_{\LSI}\le \limsup_{N\to\infty} \frac{\overline{\mathcal{I}}(P^N|M_N)}{\overline{\mathcal{E}}(P^N|M_N)}=\frac{D(\rho)}{E^{MF}[\rho]-\inf E^{MF}}.
$$
Taking the infimum over all $\rho \notin \mathcal{K}$, we obtain
$$
\limsup_{N\to\infty}\lambda^N_{\LSI}\le \lambda^\infty_{\LSI}=\inf_{\rho\notin \mathcal{K}}\frac{D(\rho)}{E^{MF}[\rho]-\inf_{\cP(\Omega)} E^{MF}}.
$$
This completes the proof of the first half of the theorem. 

Next, we consider the case when there exists $\rho_*\in\mathcal{P}(\Omega)$ a non-minimising steady state, that is to say
$$
D(\rho_*)=0\qquad\mbox{and}\qquad E^{MF}[\rho_*]>\inf_{\cP(\Omega)} E^{MF}.
$$
Consider now the sequence $\rho_*^{\otimes N} \in \cP_{\sym}(\Omega^N)$.  We have
\begin{align}
    \lambda_{N}^{\LSI} \leq& \frac{\overline{\cI}(\rho_*^{\otimes N}| M_N)}{\overline{\cE}(\rho_*^{\otimes N}| M_N)}
    \label{eq:LSchaos}
\end{align}
By \cref{thm:MS} and \cref{cor:convergence}, we can pass to the limit in the denominator to obtain
$$
\lim_{N\to\infty}\overline{\cE}(\rho_*^{\otimes N}| M_N)=E^{MF}[\rho_*]-\inf_{\cP(\Omega)} E^{MF}>0.
$$
Thus, we have that for $N$ large enough
\begin{equation}
\lambda_{N}^{\LSI}  \leq \frac{2}{E^{MF}[\rho_*]-\inf_{\cP(\Omega)} E^{MF}} \overline{\cI}(\rho_*^{\otimes N}| M_N)\, .
\end{equation}
The desired bound
$$
\lambda_{N}^{\LSI}\le \frac{C}{N}
$$ 
follows from the decay estimate for the Fisher information proved in~\cref{lem:fisher}.
\end{proof}
\begin{lemma}\label{lem:fisher}
Under \cref{A1,A2}, assume that $\rho_*$ is a critical point of the mean field free energy $E^{MF}$, then we have the following bound
\begin{eqnarray*}
\mathcal{I}(\rho_*^{\otimes N}|M_N) &= &\left(1-\frac{1}{N}\right)\int_{\Omega}   \left(|\nabla_1 W|^2\star\rho_*(x_1)-\abs*{\nabla_1 W\star\rho_*(x_1)}^2\right)\rho_*(x_1)\;\dx{x_1}\\
&&+\frac{1}{N}\int \abs*{\nabla_1 W\star\rho_*(x_1)}^2\rho_*(x_1)\dx{x_1}\\
&&\le C\, .
\end{eqnarray*}
\end{lemma}

\begin{proof}[Proof of \cref{lem:fisher}]
We start by expanding $\cI(\rho_*^{\otimes N}|M_N)$ as follows
\begin{align}
\mathcal{I}(\rho_*^{\otimes N}|M_N)=&\int_{\Omega^N} \left|\nabla \left(\log \rho_*^{\otimes N}+\beta\sum_i V(x_i)+\frac{\beta}{2N}\sum_{i,j}W(x_i,x_j)  \right)\right|^2\rho_*^{\otimes N}\;\dx{x}\\
=&  \sum_i \int_{\Omega^N} \left|\frac{\nabla \rho_*(x_i)}{\rho_*(x_i)}+\beta \nabla V(x_i)+\frac{\beta}{N}\sum_{j}\nabla_1 W(x_i,x_j)  \right|^2\rho_*^{\otimes N}\; \dx{x}\\
 =& N \int_{\Omega^N} \left|-\beta \nabla_1 W \star\rho_*(x_1)+\frac{\beta}{N}\sum_{j}\nabla_1 W(x_1,x_j)  \right|^2\rho_*^{\otimes N}\;\dx{x}\\
 =&N\beta^2 \int_{\Omega^N}  \left(\abs*{\nabla_1 W \star \rho_*(x_1)}^2-\frac{2}{N} \nabla_1W \star \rho_*(x_1) \cdot \sum_j \nabla_1W(x_1,x_j)\right.\\
 &\left.\qquad\qquad\qquad\qquad\qquad+\frac{1}{N^2}\sum_{j,k}\nabla_1W(x_1,x_j)\cdot \nabla_1W(x_1,x_k)
\rho_*^{\otimes N}\right)\;\dx{x},
\end{align}
where we have used symmetry of the particle system to write the integrand in terms of the variable $x_1$ and that $\rho_*$ is critical point of the mean field energy $E^{MF}$ which itself implies that
$$
\beta^{-1}\log\rho_* +W\star\rho_*+V=C.
$$
Next, proceeding term by term, we notice the following cancellation (for simplicity, we assume that $W(x,x)=0$, otherwise we can change $V$ by an additive constant such that this holds):
\begin{eqnarray*}
\int_{\Omega^N}  \abs*{\nabla_1 W\star\rho_*(x_1)}^2\rho_*^{\otimes N}\;\dx{x} &= &\int_\Omega  \abs*{\nabla_1 W\star\rho_*(x_1)}^2\rho_*(x_1)\;\dx{x_1},
\\
-\int_{\Omega^N} \frac{2}{N} \nabla_1 W\star\rho_*(x_1) \cdot \sum_j \nabla_1 W(x_1,x_j)\rho_*^{\otimes N}\;\dx{x}&=&-\left(2-\frac{2}{N}\right)\int_\Omega  \abs*{\nabla _1W\star\rho_*(x_1)}^2\rho_*(x_1)\;\dx{x_1},
\end{eqnarray*}
and
\begin{align}
\int_{\Omega^N} \frac{1}{N^2}\sum_{j,k}\nabla_1W(x_1,x_j)\cdot \nabla_1W(x_1,x_k)
\rho_*^{\otimes N}\;\dx{x}=&\frac{(N-1)(N-2)}{N^2}\int_\Omega  \abs*{\nabla_1W\star\rho_*(x_1)}^2\rho_*(x_1)\;\dx{x_1} \\&+\frac{N-1}{N^2}\int \abs*{\nabla_1 W}^2\star\rho_*(x_1)\rho_*(x_1)\dx{x_1} \, .
\end{align}
Putting the previous identities together, we obtain
\begin{align}
\mathcal{I}(\rho_\beta^{\otimes N}|M_N)=&\left(1-\frac{1}{N}\right)\int_{\Omega}   \left(\abs*{\nabla_1 W}^2\star\rho_*(x_1)-\abs*{\nabla_1 W\star\rho_*(x_1)}^2\right)\rho_*(x_1)\;\dx{x_1}\\&+\frac{1}{N}\int \abs*{\nabla_1 W\star\rho_*(x_1)}^2\rho_*(x_1)\dx{x_1} \, .
\end{align}
Finally, the second inequality in the statement follows from applying Jensen's inequality to obtain
$$
\abs*{\nabla_1 W}^2\star\rho_*(x)\ge \abs*{\nabla_1 W\star\rho_*(x)}^2\qquad\mbox{for every $x \in \Omega$.}
$$
The final bound follows from the assumption that $W$ grows at most polynomially (cf. ~\cref{A2}) and the fact that all steady states have finite moments of every order by~\cref{prop:steadystate}.
\end{proof}

\section{Proof of \cref{thm:contraction}}\label{sec:contraction}
\begin{proof}[Proof of \cref{thm:contraction}]
We prove the statement of the theorem by contradiction. To this end, we assume that there exists a sequence $N_i\to\infty$ and a sequence of times $t_i>0$
\begin{equation}\label{boundont}
0<t_{i}<\frac{2}{\lambda^\infty_{\LSI}}\log((E^{N_i}[\rho_{\mathrm{in}}^{\otimes N_i}]-E^{N_i}[M_{N_i}])/\eps)    
\end{equation}
such that
$$
\overline{\mathcal{E}}(\rho^{N_i}(t_i)|M_{N_i})>\max(\eps, e^{-\frac{\lambda_{\LSI}^\infty t_i}{2}}(E^{N_i}[\rho_{\mathrm{in}}^{\otimes N_i}]-E^{N_i}[M_{N_i}]).
$$
Using the finite energy assumption on the initial condition \eqref{hyp:contraction}, and \cref{A1,A2} which imply that $E^N$ is uniformly bounded below, we have that
\begin{equation}\label{b1}
\sup_i t_i\le \sup_{i}\frac{2}{\lambda^\infty_{LSI}}\log((E^{N_i}[\rho_{\mathrm{in}}^{\otimes N_i}]-E^{N_i}[M_{N_i}])/\eps)=:T(\eps)<\infty.    
\end{equation}

By the relative entropy dissipation estimate \eqref{eq:dissipation}, we obtain
$$
\overline{\mathcal{E}}(\rho^{N_i}(t_i)|M_{N_i})>e^{-\frac{\lambda_{\LSI}^\infty t_i}{2}}(E^{N_i}[\rho_{\mathrm{in}}^{\otimes N_i}]-E^{N_i}[M_{N_i}])
$$
which implies that there exists at least one $s_i\in(0,t_i)$ such that
$$
\frac{\beta^{-1}\overline{\mathcal{I}}(\rho^{N_i}(s_i)|M_{N_i})}{\overline{\mathcal{E}}(\rho^{N_i}(s_i)|M_{N_i})}<\frac{\lambda_{\LSI}^\infty}{2}.
$$
The contradiction will arise if we can show that 
$$
\liminf_{i\to\infty} \frac{\beta^{-1}\overline{\mathcal{I}}(\rho^{N_i}(s_i)|M_{N_i})}{\overline{\mathcal{E}}(\rho^{N_i}(s_i)|M_{N_i})}\ge \lambda^\infty_{\LSI}=\inf_{\rho \notin \mathcal{K}} \frac{D(\rho)}{E^{MF}[\rho]-\inf_{\mathcal{P}(\Omega)} E^{MF}}.
$$

First, we consider the case when
$$
\liminf_{i\to\infty}\overline{\mathcal{I}}(\rho^{N_i}(s_i)|M_{N_i})=\infty \, .
$$
By using the monotonicity of the relative entropy \eqref{eq:dissipation} and the finite energy assumption in the statement of the theorem, we obtain
$$
\overline{\mathcal{E}}(\rho^{N_i}(s_i)|M_{N_i})\le \overline{\mathcal{E}}(\rho_{\mathrm{in}}^{\otimes N_i}|M_{N_i})= E^{MF}[\rho_{\mathrm{in}}]-E^N[M_N]<\infty.
$$
Hence, we can conclude
$$
\liminf_{i\to\infty} \frac{\beta^{-1}\overline{\mathcal{I}}(\rho^{N_i}(s_i)|M_{N_i})}{\overline{\mathcal{E}}(\rho^{N_i}(s_i)|M_{N_i})}=\infty>\lambda^\infty_{\LSI}.
$$

Therefore, we can assume that up to a subsequence, which we do not relabel,
$$
\sup_i \overline{\mathcal{I}}(\rho^{N_i}(s_i)|M_{N_i})<\infty.
$$
By the bounded higher moment hypothesis in~\eqref{hyp:contraction}, the uniform boundedness of $s_i$~\eqref{b1}, and the propagation of moments along the flow for $K$-convex potentials~\cite{CARRILLO2020108734}, we have that $\rho^{N_i}(s_i)$ has uniformly bounded moments of order $2+\delta$. By \cref{prop:compactness}, we obtain that there exists $P_*\in\mathcal{P}(\mathcal{P}(\R^d))$, such that up to a (not relabeled) subsequence, $\rho^{N_i}(s_i)\to P_*$ in the sense of \cref{def:metricDFHS}. By the HWI inequality \cref{thm:HWI} and \cref{cor:convergence}, we obtain the strong convergence in the relative entropy term, which yields
$$
\eps\le\lim_{i\to\infty} \overline{\mathcal{E}}(\rho^{N_i}(s_i)|M_{N_i})= \int_{\mathcal{P}(\R^d)} E^{MF}[\rho]-\inf E^{MF} \;\dx{P_*}(\rho).
$$
Combining this limit with the $\liminf$-inequality of \cref{thm:MS} for the dissipation we obtain that
$$
\liminf_{i\to\infty} \frac{\beta^{-1}\overline{\mathcal{I}}(\rho^{N_i}(s_i)|M_{N_i})}{\overline{\mathcal{E}}(\rho^{N_i}(s_i)|M_{N_i})}\ge \frac{\int_{\mathcal{P}(\R^d)} D(\rho) \;\dx{P_*}(\rho)}{\int_{\mathcal{P}(\R^d)} E^{MF}[\rho]-\inf E^{MF} \;\dx{P_*}(\rho)}\ge \lambda^\infty_{\LSI},
$$
where the last inequality follows from the point-wise inequality
$$
D(\rho)\ge \lambda^\infty_{\LSI} (E^{MF}[\rho]-\inf E^{MF}) \, .
$$
This is the desired contradiction and the result now follows.
\end{proof}

\section{Proof of \cref{thm:uniquelimit}}\label{sec:uniquelmit}

\begin{proof}[Proof of \cref{thm:uniquelimit}]
The uniqueness of the minimiser in the limit follows from applying the Talagrand inequality \eqref{talagrand}, which states that the energy grows quadratically around the Gibbs measure $M_N$. We take $\rho_1$ and $\rho_2$, two minimisers of $E^{MF}$, and show that they must coincide. By the triangle and Talangrand inequality, we have
\begin{align}
d^2_2(\rho_1,\rho_2)=\overline{d}^2_2(\rho_1^{\otimes N},\rho_2^{\otimes N})
&\le 2\overline{d}^2_2(\rho_1^{\otimes N},M_N)+2\overline{d}^2_2(M_N,\rho_2^{\otimes N})\\
&=\frac{4}{\lambda^N_{\LSI}}(E^{MF}[\rho_1]+E^{MF}[\rho_2]-2E^N[M_N]).    
\end{align}
The fact that $\rho_1$ is equal to $\rho_2$ follows from taking the limit in the previous inequality, using the hypothesis that $\limsup_{N\to\infty}\lambda^N_{\LSI}>0$, and \cref{cor:convergence} to obtain
$$
\lim_{N\to\infty }E^N[M_N]=E^{MF}[\rho_1]=E^{MF}[\rho_2]=\inf_{\cP(\Omega)} E^{MF}.
$$
The quantitative convergence of $M_N$ to $\rho_\beta$, follows from the bound in~\cref{lem:rightscaling}.
\end{proof}
\begin{lemma}\label{lem:rightscaling}
Under \cref{A1,A2}, assume that  $\liminf_{N\to\infty}\lambda_{\LSI}^N=:\lambda^\infty>0$. Consider $M_N$ the $N$ particle Gibbs measure and $\rho_\beta$ the unique minimiser of the mean field energy. Then, for $N$ large enough,
$$
\overline{d}_2(M_N,\rho_\beta^{\otimes N})\le \frac{2}{\lambda^\infty}\frac{\ds\int_\Omega |\nabla_1 W|^2\star \rho_\beta\rho_\beta\;\dx{x}}{\sqrt{N}}\le \frac{C}{\sqrt{N}}.
$$
\end{lemma}
\begin{proof}[Proof of \cref{lem:rightscaling}]
Using the Talagrand and log Sobolev inequality, we have for $N$ large enough
$$
\overline{d}^2_2(M_N,\rho_\beta^{\otimes N})\le\frac{2}{\lambda_{\LSI}^N} \frac{\mathcal{E}(\rho_\beta^{\otimes N}|M_N)}{N}\le\frac{2}{(\lambda_{\LSI}^N)^2}\frac{\mathcal{I}(\rho_\beta^{\otimes N}|M_N)}{N}\le \frac{4}{(\lambda^\infty)^2} \frac{\mathcal{I}(\rho_\beta^{\otimes N}|M_N)}{N}.
$$
The result now follows from bounding the Fisher information using~\cref{lem:fisher}.
\end{proof}

\section{Proof of \cref{thm:uniform}}\label{sec:uniform}
We start by revisiting the classical propagation of chaos results~\cite{sznitman1991topics,malrieu2003convergence} by using a convexity approach based on the 2-Wasserstein distance.

\begin{theorem}\label{thm:shortime}
Under \cref{A1,A2}, if $K_V+K_W(1-1/N)\ne 0$, then
\begin{equation}
\overline{d}_2(\rho^N(t),\rho^{\otimes N}(t))\le\frac{1-e^{-\frac{K_V+K_W(1-1/N)}{2}t}}{K_V+K_W(1-1/N)}\left(\frac{\sup_{s\in[0,t]}\left(\int_\Omega |\nabla_1 W|^2\star\rho(s)\rho(s)\;\dx{x}\right)^{1/2}}{N^{1/2}}\right)\label{desired1},
\end{equation}
else if $K_V+K_W(1-1/N)= 0$, then
\begin{equation}
\overline{d}_2(\rho^N(t),\rho^{\otimes N}(t))\le \frac{t}{2}\left(\frac{\sup_{s\in[0,t]}\left(\int_\Omega |\nabla_1 W|^2\star\rho(s)\rho(s)\;\dx{x}\right)^{1/2}}{N^{1/2}}\right)\label{desired2},
\end{equation}
\end{theorem}
\begin{proof}[Proof of \cref{thm:shortime}]
Using \cite[Theorem 8.4.7]{ambrosio2008gradient}, we differentiate the 2-Wasserstein distance between $\rho^N$ and  $\rho^{\otimes N}$ along their respective flows \eqref{eq:forwardK} and \eqref{McKean}, to obtain
\begin{align}
\frac{\dx{}}{\dx{t}}\overline{d}_2^2(\rho^N,\rho^{\otimes N})&=-\frac{1}{N}\int_{\Omega^N\times \Omega^N} (x-y)\cdot \bigg(\nabla_x\left(\beta^{-1}\log\rho^N(x)+ H_N(x)\right)\\
&\qquad-\nabla_y\left(\sum_{i=1}^N\beta^{-1}\log\rho(y_i)+V(y_i)+ W\star\rho(y_i)\right)\bigg)\;\dx{\Pi}(x,y)\label{A}
\end{align}
where $\Pi\in\mathcal{P}(\Omega^N\times\Omega^N)$ denotes the optimal transport plan between $\rho^N$ and $\rho^{\otimes N}$. The convexity along 2-Wasserstein geodesics of the entropy functional as discussed in~\cite{mccann1997convexity} implies that (cf. \cite[Section 10.1.1]{ambrosio2008gradient})
\begin{equation}
\int_{\Omega^N\times \Omega^N} (x-y)\cdot (\beta^{-1}\nabla_x\log\rho^N(x)-\beta^{-1}\nabla_y\log\rho^{\otimes N}(y))\;\dx{\Pi}(x,y)\ge 0,\label{B}    
\end{equation}
where we have used that $\nabla\log\rho^N$ and $\nabla\log\rho^{\otimes N}$ are in the sub-differential of the entropy at $\rho^N$ and $\rho^{\otimes N}$, respectively (cf. \cite[Theorem 10.4.6]{ambrosio2008gradient}).

Applying inequality \eqref{B} to \eqref{A}, we obtain
\begin{align}
\frac{\dx{}}{\dx{t}}\overline{d}_2^2(\rho^N,\rho^{\otimes N})&\le-\frac{1}{N}\int_{\Omega^N\times \Omega^N} (x-y)\cdot \bigg(\nabla_xH_N(x)-\nabla_yH_N(y)\\
&\qquad\qquad+\nabla_y\left(\frac{1}{2N} \sum_{i,j=1}^NW(y_i,y_j)-\sum_{i=1}^NW\star\rho(y_i)\right)\bigg)\;\dx{\Pi}(x,y)\\
&\le -(K_V+K_W(1-1/N))\underbrace{\frac{1}{N}\int_{\Omega^N\times \Omega^N} |x-y|^2\;\dx{\Pi}(x,y)}_{=\overline{d}_2^2(\rho^N,\rho^{\otimes N})}\\
&\qquad-\underbrace{\frac{1}{N}\int_{\Omega^N\times \Omega^N}(x-y)\cdot \nabla_y \bigg(\frac{1}{2N}\sum_{i=1}^N\sum_{j=1}^N W(y_i,y_j)-\sum_{i=1}^N W\star\rho(y_i)\bigg)\;\dx{\Pi}(x,y)}_{\mathcal{R}},\label{B1}
\end{align}
where the last inequality follows from the convexity hypothesis on the potentials (cf. \cref{A1,A2}). To estimate the second term $\mathcal{R}$, we employ Cauchy-Schwarz inequality and use the fact that $\Pi$ is the optimal transference plan to obtain
\begin{align}
\mathcal{R} \le \overline{d}_2(\rho^N,\rho^{\otimes N})\left(\underbrace{\frac{1}{N}\int_{\Omega^N}\left|\nabla \left(\frac{1}{2N}\sum_{i=1}^N\sum_{j=1}^N W(y_i,y_j)-\sum_{i=1}^N W\star\rho(y_i)\right)\right|^2\;\dx{\rho}^{\otimes N}}_{I} \right)^{1/2}\label{C}   
\end{align}
Expanding the square, using the symmetry of the underlying system and the fact that $W(y,y)=0$, we obtain 
\begin{align}
I&=\int_{\Omega^N}\left|\frac{1}{N}\sum_{j=2}^N\nabla_1 W(y_1,y_j)-\nabla_1 W\star\rho(y_1)\right|^2\;\dx{\rho}^{\otimes N}\\    
&=\int_{\Omega^N}\underbrace{\frac{1}{N^2}\sum_{j=2}^N\sum_{k=2}^N\nabla_1 W(y_1,y_j)\nabla_1 W(y_1,y_k)}_{A}-\underbrace{\frac{2}{N}\sum_{j=2}^N\nabla_1 W(y_1,y_j)\nabla_1W\star\rho(y_1) }_{B}\\
&\qquad\qquad+(\nabla_1 W\star\rho(y_1))^2\;\dx{\rho}^{\otimes N}.
\end{align}
Going term by term, we have
$$
A=\frac{(N-1)(N-2)}{N^2}\int_\Omega (\nabla_1 W\star\rho(y_1))^2\;\dx{\rho}(y_1)+\frac{(N-1)}{N^2}\int_\Omega |\nabla_1 W|^2\star\rho(y_1)\;\dx{\rho}(y_1),
$$
$$
B=-2\frac{N-1}{N}\int_\Omega (\nabla_1 W\star\rho(y_1))^2\;\dx{\rho}(y_1).
$$
Using these identities, we are left with
\begin{align}
I&=\frac{1}{N}\left(1-\frac{1}{N}\right)\int_\Omega |\nabla_1 W|^2\star\rho(y_1)-(\nabla_1 W\star\rho(y_1))^2\;\dx{\rho}(y_1)+\frac{1}{N^2}\int_{\Omega}(\nabla_1 W\star\rho(y_1))^2\;\dx{\rho}(y_1)\\
&\le \frac{1}{N} \int_\Omega |\nabla_1 W|^2\star\rho(y_1)\;\dx{\rho}(y_1),
\end{align}
where the last inequality follows from Jensen's inequality. Replacing the previous equation in \eqref{C}, combined with \eqref{B1}, we obtain
\begin{align}
\frac{\dx{}}{\dx{t}}\overline{d}_2^2(\rho^N,\rho^{\otimes N})\le& -(K_V+K_W(1-1/N)) \overline{d}_2^2(\rho^N,\rho^{\otimes N})\\&+\overline{d}_2(\rho^N,\rho^{\otimes N})\frac{\left(\int_\Omega |\nabla_1 W|^2\star\rho \,\rho\;\dx{x}\right)^{1/2}}{N^{1/2}}.
\end{align}
The estimates \eqref{desired1} and \eqref{desired2} now follow from Gr\"onwall's inequality.
\end{proof}
We are now ready to prove \cref{thm:uniform}.

\begin{proof}[Proof of \cref{thm:uniform}]
Using the uniform integrability assumption on the gradients of $W$ \eqref{hypintegrability}, the estimates of \cref{thm:shortime} simplify to
\begin{equation}
\overline{d}_2(\rho^N(t),\rho^{\otimes N}(t))\le \frac{C}{N^{1/2}}
\begin{cases}
\frac{1-e^{t(K_V+K_W(1-1/N))/2}}{K_V+K_W(1-1/N)}&\mbox{if $K_V+K_W(1-1/N)\ge 0$}\\
\frac{t}{2}&\mbox{if $K_V+K_W(1-1/N)= 0$}\\
\frac{e^{t|K_V+K_W(1-1/N)|/2}-1}{|K_V+K_W(1-1/N)|}&\mbox{if $K_V+K_W(1-1/N)\le 0$}
\end{cases}\label{shortime}
\end{equation}

Next, we derive a competing estimate by employing the triangle inequality and the long time behavior of the flows. More specifically,
\begin{align}
\overline{d}_2(\rho^N(t),\rho^{\otimes N}(t))\le & \;\overline{d}_2(\rho^N(t),M_N)+\overline{d}_2(M_N,\rho_\beta^{\otimes N})+\overline{d}_2(\rho_\beta^{\otimes N},\rho^{\otimes N}(t)) \\
=&\; \overline{d}_2(\rho^N(t),M_N)+\overline{d}_2(M_N,\rho_\beta^{\otimes N})+d_2(\rho_\beta,\rho(t)) \, .
\label{3triangle}
\end{align}
For the first term, we use the Talagrand inequality \eqref{talagrand}
$$
\overline{d}_2(\rho^N(t),M_N)\le \left(\frac{2}{\lambda_{\LSI}^N}\right)^{1/2}\overline{\mathcal{E}}^{1/2}(\rho^N(t)|M_N).
$$
By the log Sobolev inequality we obtain exponential contraction of the relative entropy 
\begin{equation}\label{firsterm}
\overline{d}_2(\rho^N(t),M_N)\le e^{-\frac{\lambda_{\LSI}^N}{2} t}\left(\frac{2}{\lambda_{\LSI}^N}\right)^{1/2}\overline{\mathcal{E}}^{1/2}(\rho_{\mathrm{in}}^{\otimes N}(t)|M_N) \le Ce^{-\frac{\lambda^\infty}{4} t},
\end{equation}
where in the last equality we have used the hypothesis $\rho_{\mathrm{in}}$ has finite energy and that the log Sobolev constant does not degenerate.

\noindent For the second term, we use Lemma \ref{lem:rightscaling} to obtain
\begin{equation}\label{secondterm}
\overline{d}_2(M_N,\rho_\beta^{\otimes N})\le \frac{C}{\sqrt{N}}.
\end{equation}

\noindent For the third term, we use the limiting Talgrand inequality and the limiting log Sobolev inequality to obtain the exponential contraction estimate
\begin{align}
d_2(\rho(t),\rho_\beta)&\le
\left(\frac{2}{\lambda_\infty}\right)^{1/2}(E^{MF}[\rho(t)]-E^{MF}[\rho_\beta])^{1/2}\nonumber \\
&\le \left(\frac{2}{\lambda_\infty}\right)^{1/2}e^{-\frac{\lambda_\infty}{2} t}(E^{MF}[\rho_{\mathrm{in}}]-E^{MF}[\rho_\beta])^{1/2}\nonumber\\
&\le Ce^{-\frac{\lambda_\infty}{2} t}.\label{thirdterm}
\end{align}
Combining \eqref{3triangle} with \eqref{firsterm}, \eqref{secondterm} and \eqref{thirdterm} we obtain the estimate
\begin{equation}\label{2}
    \overline{d}_2(\rho^N(t),\rho^{\otimes N}(t))\le C\left(e^{-\frac{\lambda_\infty}{4}t}+\frac{1}{\sqrt{N}}\right).
\end{equation}

The result now follows from interpolating the estimates \eqref{shortime} and \eqref{2}. In the case, $K_V+K_W(1-1/N)>0$ the desired estimate follows directly from \eqref{shortime}. For $K_-:=K_V+K_W(1-1/N)<0$, we consider the distinguished time scale $T_N:= \log N^\gamma$, for some $\gamma>0$ to be chosen in terms of $K_-$. Applying~\eqref{2}, we obtain
\begin{align}
\overline{d}_2(\rho^N(t),\rho^{\otimes N}(t)) \leq C\left(N^{-\gamma\frac{\lambda_\infty}{4}}+N^{-\frac{1}{2}}\right)\qquad\mbox{for $t>T_N$}.
\end{align}
For $t < T_N$, we apply~\eqref{shortime}1 to obtain
\begin{align}
\overline{d}_2(\rho^N(t),\rho^{\otimes N}(t)) \leq CN^{-\frac{1+\gamma K_-}{2}} \, .
\end{align}
Choosing 
$$
\gamma=\left(\frac{\lambda_\infty}{2}-K_-\right)^{-1},
$$ 
we obtain that for every $t\in(0,\infty)$ 
$$
\overline{d}_2(\rho^N(t),\rho^{\otimes N}(t))\le \frac{C}{N^\theta}
$$
is satisfied with
$$
\theta=\frac{1}{2}\frac{\lambda_\infty}{\lambda_\infty-2K_-}. 
$$
\end{proof}

\section{Proof of~\cref{thm:twoscale}}\label{sec:twoscale}

As mentioned earlier, our proof of~\cref{thm:twoscale} will rely on the two-scale approach to log Sobolev inequalities introduced in~\cite{OR07} and discussed further in~\cite{GOVW09}, see also~\cite{Lel09}. Before we introduce the main result of~\cite{OR07}, we introduce some preliminary notions. 
\begin{definition}[Conditional measures]
Given a probability measure  $\mu_N \in \cP(\Omega^N) $ we define the conditional measure $\mu_{N,i}, i \in{1,\dots,N}$ as the family of measures indexed by $x_j,\, j \neq i$ such that for all $\varphi \in C_b(\Omega^N)$
\begin{align}
\int_{\Omega^N}\varphi \dx{\mu_N} = \int_{\Omega^{N-1}} \int_{\Omega} \varphi \dx{\mu_{N,i}} \dx{\mu_{N\setminus\set{i}}} \, ,
\end{align} 
where $\mu_{N\setminus\set{i}}$ is the marginal of $\mu_N$ obtained by integrating out $x_i \in \Omega$.\label{def:cond}
\end{definition}

We can now the state the result of interest.
\begin{theorem}[{\cite[Theorem 1]{OR07}}]
Let $\Omega$ be a smooth, connected, and complete Riemannian manifold and assume that the measure $\mu_N$ has a Gibbs structure, that is to say
\begin{align}
\mu_N(\dx{x})= Z_{N}^{-1}e^{-\beta H_N} \dx{x} \, ,
\end{align}
where $\dx{x}$ is the Riemannian volume measure and $H_N: \Omega^N \to \R$ is some smooth Hamiltonian. Assume there exists some constants $\kappa_{ij}$, such that for all $i \neq j$
\begin{align}
\norm{D^2_{x_i x_j} H_N} \leq \kappa_{ij} \qquad\mbox{for all $x \in \Omega ^N$} ,
\end{align}
where $\norm{\cdot}$ is the operator norm of $D^2_{x_ix_j}H_N$. Furthermore, assume that the conditional measures $\mu_{N,i}$ satisfy a log Sobolev inequality with uniform constant $\lambda_{\LSI}^{N,i}$ for all $\hat{x} \in \Omega^{N-1}$. Consider the matrix $A \in \R^{N \times N}$ with entries $A_{ii}=\lambda^{N,i}_{\LSI}$ and $A_{ij}=-\beta \kappa_{ij}$; if 
\begin{align}
A \geq C I^{N\times N} \, ,
\end{align}
the measure $\mu_N$ satisfies a log Sobolev inequality with constant $C$.
\label{thm:or}
\end{theorem}

Relying on \cref{thm:or}, we present now the proof of~\cref{thm:twoscale}.
\begin{proof}[Proof of~\cref{thm:twoscale}]
We note first that from the \cref{def:cond}, the conditional measure  $M_{N,i}$ of $M_N$ can be expressed as 
\begin{align}
M_{N,i}(\dx{x_i})=& M_N(\dx{x_i}| x_1, \dots ,x_{i-1} ,x_{i+1},\dots, x_N)=\frac{M_N}{(M_N)_{N\setminus\set{i}}} \, ,
\end{align}
where $(M_N)_{N\setminus\set{i}}$ is the marginal of $M_N$ obtained by integrating out $x_i$. We are thus left with
\begin{align}
M_{N,i} =& \frac{ \exp\bra*{-\beta (V(x_i) + \frac{1}{N}\sum_{j=1}^N W(x_i,x_j) ) -\beta (\sum_{j\neq i ,j=1}^N V(x_j)  +\frac{1}{2N}\sum_{j,k=1,j,k\neq i} W(x_j,x_k))}}{\int_{\Omega} \exp\bra*{-\beta H_N} \dx{x_i}} \\
=&Z_{N,i}^{-1}  \exp\bra*{-\beta (V(x_i) + \frac{1}{N}\sum_{j=1}^N W(x_i,x_j) )} \, ,
\end{align}
where
\begin{align}
Z_{N,i}=\int_{\Omega} \exp\bra*{-\beta (V(x_i) + \frac{1}{N}\sum_{j=1}^N W(x_i,x_j) )} \dx{x_i} \, .
\end{align}
We now assert that the conditional measure $M_{N,i}$ satisfies a log Sobolev inequality. We first treat the case in which $\Omega$ is compact, that is to say~\cref{thm:twoscale}~\ref{twoscalea}. By the Holley--Stroock perturbation Theorem~\cite[Proposition 5.1.6]{bakry2013analysis}, we have that 
\begin{align}
\lambda^{N,i}_{\LSI}\geq  e^{- 2 \beta  (\norm{W}_{\Leb^\infty(\Omega^2)} + \norm{V}_{\Leb^\infty(\Omega^)})} \lambda_{\LSI}^\Omega \, ,
\end{align}
for all $x_j, j\neq i, j=1,\dots,N$ and where $\lambda_{\LSI}^\Omega$ is the optimal log Sobolev constant of the Lebesgue measure on $\Omega$. 

 Note that because of the exchangeability of the underlying particle system,  we have that $\lambda^{N,i}_{\LSI}=\lambda^{N,j}_{\LSI}$ for all $i,j = 1,\dots,N$. 

Note now that
\begin{align}
D^2_{x_i x_j}H_N(x_1,\dots,x_N) = \frac{1}{N}D^2_{x_i x_j} W(x_i,x_j) \, . 
\end{align}
Using the hypothesis that $W\in W^{2,\infty}(\Omega^2)$, we can bound
\begin{align}
\norm{D^2_{x_i x_j}H_N}_{\Leb^\infty(\Omega^N)} \leq \frac{1}{N} \norm{D^2_{x y} W}_{\Leb^\infty(\Omega^2)} =: \kappa_{ij} \, ,
\end{align} 
for all $i,j=1,\dots,N$. We will show that the matrix $A \in \R^{N \times N}$ from \cref{thm:or} is positive definite, by showing that it is diagonally dominant. In fact, for $\beta$ sufficiently small we have that
\begin{align}
   A_{ii}-\sum_{j\ne i}|A_{ij}|\ge e^{- 2 \beta  (\norm{W}_{\Leb^\infty(\Omega^2)} + \norm{V}_{\Leb^\infty(\Omega^)})} \lambda_{\LSI}^\Omega-\beta \frac{N-1}{N}   \norm{D^2_{x y} W}_{\Leb^\infty(\Omega^2)} >c >0\, ,
\end{align}
holds true for all $N$ with the constant $c$ independent of $N$. Applying~\cref{thm:or},  \cref{thm:twoscale}~\ref{twoscalea} now follows. 

For the proof of~\cref{thm:twoscale}~\ref{twoscaleb} we can apply essentially the same perturbative argument as before but now around the measure $Z_{V}^{-1}e^{-V} \dx{x}$. For $\epsilon$ sufficiently small, we obtain that the analogous bound
\begin{align}
  A_{ii}-\sum_{j\ne i}|A_{ij}|\ge e^{- 2\epsilon \beta  \norm{W}_{\Leb^\infty(\Omega^2)} } \lambda_{\LSI}^V-\epsilon \beta \frac{N-1}{N}   \norm{D^2_{x y} W}_{\Leb^\infty(\Omega^2)} >c >0\, 
\end{align}
holds true for all $N$ with the constant $c$ independent of $N$.
\end{proof}

\section{Proof of \cref{thm:fluctuations}}\label{sec:fluctuations}
To simplify the computations in this section, we will take the following definition of the negative Sobolev norm $H^{-s}_0(\T^d)$ of mean-zero distributions is given by

\begin{equation}
\|h\|_{H^{-s}(\T^d)}^2:=\sum_{j\in \N}|\langle h, \phi_j \rangle|^2,
\label{eq:snorm}
\end{equation}
where $(\phi_j)_{j\in\N}$ is a given a smooth orthonormal basis for the Sobolev space $H^{s}_0(\T^d)$.

We remark that Theorem~\ref{thm:fluctuations} can also be proved when $\Omega  = \R^d$, under appropriate assumptions on the confining and interaction potentials. In particular, we can construct an appropriate orthonormal basis using the eigenfunctions of the linearised McKean-Vlasov operator with a weighted inner-product. Conditions on the growth of the confining potential so that the Sobolev embedding theorems needed in the proof in the appropriate weighted spaces are given in~\cite{Lorenzi}.

\begin{lemma}[Law of large numbers] Assume that $\liminf_{N\to\infty}\lambda^N_{\LSI}>0$. 
Let $\mu^{(N)}$ be the empirical measure associated to the $N$-particle Gibbs measure $M_N \in \cP_{\mathrm{sym}}((\T^d)^N)$. Then, for any $s>\frac{d+2}{2}$, there exists $C>0$ such that
\[
\mathbb{E}\pra*{\|\mu^{(N)}-\rho_\beta\|_{H^{-s}(\T^d)}^2}\le \frac{C}{N} \, ,
\]
where $\rho_\beta \in \cP(\T^d)$ is the unique critical point of the mean field energy $E^{MF}$.
\end{lemma}
\begin{remark}
The solutions to the linear SPDE \eqref{SPDE1} are supported in $H^{-s}$ with $s>d/2$, see \cref{mild}. Hence, the Law of large numbers does not hold for any $H^{-s}$ with $s< d/2$. 
\end{remark}

\begin{proof}
We consider $\rho_\beta^{(N)}$ the empirical measure associated to $\rho_\beta^{\otimes N} \in \cP_{\mathrm{sym}}((\T^d)^N)$, i.e. $\rho_\beta^{(N)}$ the probability-measure valued random variable defined as 
$$
\rho_\beta^{(N)}=\frac{1}{N}\sum_{i=1}^N\delta_{X_i}\qquad 
\mbox{such that $(X_1,...,X_N)$ are distributed according to $\rho_\beta^{\otimes N}$} \, .
$$
By the triangle inequality we have that
\begin{equation}\label{eq:decompose}
    \mathbb{E}\pra*{\|\mu^{(N)}-\rho_\beta\|^2_{H^{-s}(\T^d)}}\le 2\bra*{\mathbb{E}\pra*{\|\mu^{(N)}-\rho_\beta^{(N)}\|^2_{H^{-s}(\T^d)}} + \mathbb{E}\pra*{\|\rho_\beta^{(N)}-\rho_\beta\|^2_{H^{-s}(\T^d)}}}
\end{equation}

We start by controlling the first term $\mathbb{E}\pra*{\|\mu^{(N)}-\rho_\beta^{(N)}\|^2_{H^{-s}(\T^d)}}$. We consider the optimal coupling between $\mu^{(N)}$ and $\rho_\beta^{(N)}$ such that
$$
\mathbb{E}\pra*{d_2^2(\mu^{(N)},\rho_\beta^{(N)})}=\mathfrak{D}_2^2\left(\hat{\mu}^N,\hat{\rho}_\beta^N\right) \, ,
$$
where $\hat{\mu}^N=\mathrm{Law}\left(\mu^{(N)}\right)\in \mathcal{P}(\mathcal{P}(\T^d))$, $\hat{\rho}_\beta^N=\mathrm{Law} \left(\rho_\beta^{(N)}\right)\in \mathcal{P}(\mathcal{P}(\T^d))$, and $\mathfrak{D}_2$ is the 2-Wassertein distance defined on the space of  probability measures of the metric space $(\mathcal{P}(\T^d),d_2)$. Then by the isometry from~\cref{thm:HM}, we obtain
$$
\mathfrak{D}_2^2\left(\hat{\mu}^N,\hat{\rho}_\beta^N\right)=\overline{d}_2^2(M_N,\rho_\beta^{\otimes N})\le \frac{C}{N},
$$
where we have used~\cref{lem:rightscaling} and the fact that $\liminf_{N\to \infty}\lambda^N_{\LSI}>0$ for the last inequality. Next, using the compactness of $\T^d$ and the fact that $s>d/2+1$, we have the pointwise bound 
$$
d_2^2(\mu^{(N)},\rho_\beta^{(N)})\ge c_1 \|\mu^{(N)}-\rho_\beta^{(N)}\|_{W^{-1,\infty}(\T^d)}^2\ge c_2 \|\mu^{(N)}-\rho_\beta^{(N)}\|^2_{H^{-s}(\T^d)} \, ,
$$
for some constants $c_1,c_2>0$ independent of $\mu^{(N)},\rho_\beta^{(N)}$. Combining the previous three expressions, we obtain the first desired bound
\begin{equation}\label{des1}
    \mathbb{E}\pra*{\|\mu^{(N)}-\rho_\beta^{(N)}\|^2_{H^{-s}(\T^d)}}\le \frac{C}{N}.
\end{equation}

Next, we bound $\mathbb{E}\pra*{\|\rho_\beta^{(N)}-\rho_\beta\|^2_{H^{-s}(\T^d)}}$. We take $\{\phi_j\}_{j=0}^\infty$ to be the orthonormal basis of $H^{s}_0(\T^d)$ from~\eqref{eq:snorm}. Expanding the square, we obtain
\begin{align}
\mathbb{E}\pra*{\|\rho_\beta^{(N)}-\rho_\beta\|^2_{H^{-s}(\T^d)}}=&\mathbb{E}\pra*{\sum_{j=1}^\infty\abs*{\langle \rho_\beta^{(N)}-\rho_\beta,\phi_j\rangle}^2}\\ 
=&\sum_{j=1}^\infty \mathbb{E}\pra*{\left|\frac{1}{N}\sum_{i=1}^N\phi_j(X_i)- \int \phi_j\rho_\beta\;\dx{x}\right|^2}\\
=&\frac{1}{N}\sum_{j=1}^\infty\int \phi_j^2\rho_\beta\;\dx{x}-\frac{1}{N}\left(\int \phi_j\rho_\beta\;\dx{x}\right)^2\\
\le&\frac{\|\rho_\beta\|_{L^\infty(\T^d)}}{N}\sum_{j=1}^\infty\|\phi_j\|_{\Leb^2(\T^d)}^2. 
\end{align}
Noticing that $s>d/2$ implies that the embedding $H^{s}_0(\T^d)\to \Leb^2(\T^d)$ is Hilbert-Schmidt, we can use the boundedness of $\rho_\beta$ to obtain
\begin{equation}\label{des2}
    \mathbb{E}\pra*{\|\rho_\beta^{(N)}-\rho_\beta\|^2_{H^{-s}(\T^d)}}\le \frac{C}{N} \, .
\end{equation}
Combining \eqref{eq:decompose} with \eqref{des1} and \eqref{des2}, we obtain the desired conclusion.
\end{proof}
\begin{remark}
An alternative proof of inequality \eqref{des2} can be found in~\cite[Theorem 5.1]{hauray2014kac}.
\end{remark}
We now consider the implications of having uniform control of the log Sobolev constant the fluctuations of the stationary solutions to~\eqref{eq:particle system}, i.e. solutions of~\eqref{eq:particle system} started at stationarity.  To this end, we consider the empirical measure process $t \mapsto \mu^{(N)}(t)$ defined by
\begin{equation}
 \mu^{(N)}(t):=\frac 1 N \sum_{i=1}^{N}\delta_{X_t^i}, \nonumber   
 \end{equation}
where $X_t=(X_t^{1},\dots,X_t^{N})$ is the solution to \eqref{eq:particle system} started from the unique invariant Gibbs measure $M_N$.  Our goal is to analyze the corresponding fluctuation process $t \mapsto \eta^{N}(t)$ defined by
\begin{equation}
\eta^{N}(t):=\sqrt{N}(\mu^{(N)}(t)-\rho_{\beta}),  \nonumber  
\end{equation}
As a direct consequence of the estimates from the previous lemma and using the fact that $\liminf_{N\to\infty}\lambda^{N}_{LS}>0$, we have the uniform bound
\begin{equation}
\sup_{N,t}\mathbb{E}\pra*{\|\eta^N(t)\|_{H^{-s}(\T^d)}^2}<\infty \qquad \mbox{for any $s>d/2+1$.}\label{L2control}
\end{equation}
In the sequel, we will use the above estimate together with the classical martingale method, c.f.~\cite[Chapter 8]{RSMUP_1982__67__171_0}, to establish convergence in law as $N \to \infty$ of $\eta^N$ to the stationary solution $\eta$ of the following linear SPDE
\begin{equation}
 \partial_t\eta=\mathcal{L}_{\rho_\beta}\eta+\nabla\cdot(\sqrt{\rho_\beta}\xi),   \label{SPDE1}
\end{equation}
where $\xi$ is a mean-zero space-time white noise on $\R_{+} \times \T^{d}$ and $\mathcal{L}_{\rho_\beta}$ is the linearisation of the McKean-Vlasov operator \eqref{McKean} around the unique invariant measure $\rho_\beta$ defined by
\begin{equation}\label{e:linearised}
\mathcal{L}_{\rho_\beta}\psi=\beta^{-1}\Delta \psi+\nabla\cdot(\nabla W\star\psi \,\rho_\beta)+\nabla\cdot(\nabla W\star\rho_\beta \,\psi)+\nabla\cdot(\nabla V\,\psi).
\end{equation}
The SPDE \eqref{SPDE1} can be solved using classical methods, the three typical notions of solution being the mild, weak, and martingale formulations.  As is typical, the martingale formulation is the most convenient for identifying the limiting law of a tight subsequence of $\eta^N$, while the mild formulation provides a clearer picture of the uniqueness in law and hence the convergence in law of the full sequence.

\medskip

Denote $U$ as the closed mean-zero subspace of $L^{2}(\T^{d};\T^{d})$ and $H:=H^{-s}(\T^{d})$ for $s>\frac{d}{2}$. Let $(\mathcal{X},\mathcal{F}, (\cF_t)_{t \geq 0},\P)$ be a stochastic basis, i.e a complete filtered probability space with a right continuous filtration endowed with an $\mathcal{F}_t$-adapted $U$-valued cylindrical Wiener process $\textbf{W}$ and let $\eta_{0}$ be an $H$-valued $\cF_0$-measurable random variable independent of $\textbf{W}$.  The mild formulation of \eqref{SPDE1} with initial condition $\eta_{0}$ involves stochastic integration in Hilbert spaces, c.f. \cite[Chapter 4]{RSMUP_1982__67__171_0}, which we quickly review for our specific case below.  Let $\mathbf{L}^{2}_{0}(U;H)$ denote the Hilbert--Schmidt operators from $U $ to $H$ equipped with the standard Hilbert--Schmidt norm $\norm{\cdot}_{\mathbf{L}_2^0(U;H)}$.  Given $T>0$ and $\Phi \in L^{2}([0,T];\mathbf{L}^{2}_{0}(U;H))$, the stochastic integral $[0,T]\ni t\mapsto \int_{0}^{t}\Phi(s)d \textbf{W}(s)$ is well defined as a continuous $\cF_t$-martingale with trajectories in $C([0,T]; H)$.  We remind the reader that the relevance of $\mathbf{L}^{2}_{0}(U;H)$ in this context is that the It\^o isometry takes the following form
\begin{equation}
\E \pra*{\left \| \int_{0}^{t}\Phi(s)d \textbf{W}(s) \right \|_{H}^{2}}=\int_{0}^{t}\| \Phi(s)\|_{\mathbf{L}^{2}_{0}(U;H) }^{2}\dx s \nonumber.
\end{equation}

\medskip

The mild solution to $t \mapsto \eta(t) \in H $ to~\eqref{SPDE1} with initial condition $\eta_{0}$ is  then given by the following stochastic convolution:
\begin{equation}
\eta^\infty(t):=e^{t\mathcal{L}_{\rho_{\beta}}}\eta_{0}+\int_{0}^{t}e^{(t-s)\mathcal{L}_{\rho_{\beta}} } \nabla \cdot \big (\sqrt{\rho_{\beta}}d\textbf{W}(s) \big)\label{eq:Mild}.
\end{equation}

\begin{lemma}\label{mild}
The mild solution \eqref{eq:Mild} is well-defined as a stochastic process with trajectories in $C([0,T];H)$. 
\end{lemma}

\begin{proof}
In light of our remarks in the preceding paragraph, it suffices to show that $s \mapsto \Phi(s)$ defined by $U \ni u \mapsto e^{(t-s)\mathcal{L}_{\rho_{\beta}} }\nabla \cdot (\sqrt{\rho_{\beta}}u ) \in H$ belongs to $L^{2}( [0,T];\mathbf{L}^{2}_{0}(U;H) )$.   Given an orthornormal basis $\{e_{k}\}_{k=1}^{\infty}$ of $U$, using the definition of the Hilbert--Schmidt norm and integrating by parts leads to
\begin{align}
\int_{0}^{T}\| \Phi(s)\|_{\mathbf{L}^{2}_{0}(U;H) }^{2}\dx s&=\sum_{k=1}^{\infty}\int_{0}^{T}\|e^{(t-s)\mathcal{L}_{\rho_{\beta}} }\nabla \cdot \big (\sqrt{\rho_{\beta}}e_{k}\big)\|_{H}^{2} \dx{s}\nonumber \\
&=\sum_{j,k=1}^{\infty}\int_{0}^{T} \langle e_{k},\sqrt{\rho_{\beta} }\nabla e^{(t-s)\mathcal{L}_{\rho_{\beta}}^{*} }\phi_{j} \rangle_{U}^{2}\dx s \nonumber \\
&=\sum_{j=1}^{\infty} \int_{0}^{T} \|\sqrt{\rho_{\beta} } \nabla e^{(t-s)\mathcal{L}_{\rho_{\beta}}^{*} }\phi_{j}\|_{U}^{2}\dx s, \label{normId}
\end{align}
where in the last step we used Parseval's identity in $U$.  Finally, we note that the coercivity hypothesis for $\mathcal{L}_{\rho_{\beta}}$ (see \cref{lem:gap}) implies that 
\begin{equation}
\int_{0}^{t} \|\nabla e^{(t-s)\mathcal{L}_{\rho_{\beta}}^{*} }\phi_{j}\|_{U}^{2}\dx{s} \leq C \|\phi_{j}\|_{L^{2}(\T^d)}^{2} \nonumber.
\end{equation}
Combining this with \eqref{normId} and using the fact that $\rho_{\beta} \in L^{\infty}(\T^{d})$, yields
\begin{equation}
\int_{0}^{T}\| \Phi(s)\|_{\mathbf{L}^{2}_{0}(U;H) }^{2} \dx{s} \leq C\|\rho_\beta\|_{L^\infty(\T^d)} \sum_{j=1}^{\infty}\|\phi_{j}\|_{L^{2}(\T^d)}^{2}. 
\label{eq:trQ}
\end{equation}
Since $s>\frac{d}{2}$, the embedding of $H^{s}_0(\T^d)$ into $L^{2}(\T^d)$ is Hilbert--Schmidt. Thus, the above series converges, completing the proof of the lemma.
\end{proof}

\medskip

\begin{remark}\label{rem:unique}
We note that the representation \eqref{eq:Mild} immediately implies that solutions to \eqref{SPDE1} are unique in law, i.e.~\eqref{SPDE1} satisfies weak uniqueness.  That is to say, given two different stochastic bases $(\mathcal{X}, \mathcal{F}, (\mathcal{F}_t)_{t \geq 0}, \P,\textbf{W})$ and $(\tilde{\mathcal{X}}, \tilde{\mathcal{F}}, (\tilde{\mathcal{F}}_t)_{t \geq 0}, \tilde{\P}, \widetilde{\textbf{W}} )$ defining solutions $\eta $ and $\tilde{\eta}$ to \eqref{SPDE1} on their respective probability spaces through the formula \eqref{eq:Mild}, the laws of $\eta$ and $\tilde{\eta}$ agree on $C([0,T];H ) $ for any $T>0$, as long as $\eta_0$ and $\tilde{\eta}_0$ are equal in law on $H$. 
\end{remark}

\medskip

For our purposes, it is easier to work with the martingale formulation of \eqref{SPDE1}, which is in turn motivated by the weak formulation of \eqref{SPDE1}.  Hence, we note in passing that the mild solution    \eqref{eq:Mild} has the property that for each $t \in [0,T]$ the following equality holds in $H^{-s}(\T^d)$, for $s>\frac{d+2}{2}$,
\begin{equation}
\eta^\infty(t)=\eta_{0}+\int_{0}^{t}\mathcal{L}_{\rho_{\beta}}\eta^{\infty}(s)
\;\dx{s}+\int_{0}^{t} \nabla \cdot(\sqrt{\rho_{\beta}}d\textbf{W}_{s}).\label{weakSol}
\end{equation}

\medskip

\begin{lemma}\label{martingalelem}
Let $(\mathcal{X}, \mathcal{F},(\mathcal{F}_{t})_{t \geq 0}, \P )$ be a filtered probability space. Assume that $\eta$ is a continuous-time  $\mathcal{F}_t$-adapted $H$-valued stochastic process and define $t \mapsto M(t) $ by
\begin{equation}
M(t):=\eta(t)-\eta_{0}-\int_{0}^{t}\mathcal{L}_{\rho_{\beta}}\eta(s)\;\dx{s}. \label{Mart0}
\end{equation}
Assume that the following two conditions hold:
\begin{itemize}

\item For all $\varphi \in C^{\infty}(\T^d)$ it holds that 
\begin{equation}
t \mapsto \langle M(t),\varphi \rangle \label{Mart1}
\end{equation}
 is an $\cF_t$-martingale.
\item For all $\varphi,\psi \in C^{\infty}(\T^d)$ it holds that
\begin{equation}
t \mapsto \langle M(t),\varphi \rangle \langle M(t),\psi \rangle-t\int_{\T^d}  \nabla \phi\cdot \nabla \psi \rho_{\beta}\;\dx{x} \label{Mart2}
\end{equation}
is an $\cF_t$ martingale.

\end{itemize}

Then, $\eta$ is equal in law on $C([0,T]; H)$ to the mild solution~\eqref{eq:Mild}.  
\end{lemma}

\begin{proof}
By \cite[Theorem 8.2]{RSMUP_1982__67__171_0}, it follows that on a suitable extension of the probability space, $\eta$ is a weak solution in the sense of \eqref{weakSol}.  By \cite[Theorem 5.4]{RSMUP_1982__67__171_0}, the weak and mild solutions coincide on that probability space, so equality in law follows from~\cref{rem:unique}.  
\end{proof}

\medskip

Using the above lemma, we are now finally in a position to prove the convergence of the fluctuations.

\begin{theorem}
Assume that $\liminf_{N\to\infty}\lambda^{N}_{LS}>0$ and that $V$ and $W$ are smooth. Then, for any $m>d/2 + 3$ the fluctuation process $\eta^{N}$ converges in law on $C([0,T];H^{-m}(\T^d))$ to the unique stationary mild solution $\eta^\infty$ of the SPDE \eqref{SPDE1}.
\end{theorem}
\begin{proof}
The proof has four steps.  In Step 1, we apply It\^o's formula to show that $\eta^{N}$ satisfies~\eqref{s1}, an approximate version of the weak formulation \eqref{weakSol}.  In Step 2, we combine Step 1 with the uniform bound \eqref{L2control} to show that the laws of $(\eta^{N})_{N \in \N}$ on $C([0,T];H^{-m}(\T^d))$ are uniformly tight for $m$ large enough.  In Step 3, we pass to the limit in the martingale problem and verify the assumptions of Lemma~\eqref{martingalelem} to identify the limit along any tight subsequence.  In Step 4, we conclude the uniqueness of the limit and hence the proof of the theorem.

\medskip

\newcounter{CLT} 
\refstepcounter{CLT} 
{\sc Step} \arabic{CLT}.\label{CLTP1}\refstepcounter{CLT} In this step, we show that for all $\varphi \in C^{\infty}(\T^{d})$ it holds
\begin{equation}
 \dx \langle\eta^{N},\varphi\rangle=\langle \eta^{N}, \mathcal{L}_{\rho_\beta}^{*} \varphi \rangle \dx t+\langle R_{N}, \varphi \rangle \dx t +\sqrt{2}(\beta N)^{-\frac{1}{2}}\sum_{i=1}^{N}\nabla\varphi(X_t^i) \cdot dB_t^i\label{s1} \, ,
\end{equation}
with $X_t^i$ the solution to~\eqref{eq:particle system} and where $R_{N}$ is defined as
\begin{equation}
R_{N}:=N^{-\frac{1}{2}}\nabla\cdot(\eta^N \nabla_{1} W\star\eta^N) \, .\label{s8}
\end{equation}
Indeed, It\^o's formula gives 
\begin{align}
\dx \varphi(X_t^i)=& \bra*{ \beta^{-1} \Delta- \nabla V \cdot \nabla } \varphi(X_t^i) \dx t -\frac{1}{N}\sum_{j=1}^{N}  \nabla_{1} W(X_t^i,X_t^j)  \cdot\nabla \varphi(X_t^i)\dx t\\&+\sqrt{2}\beta^{-1}\nabla \varphi(X_t^i) \cdot dB_t^i \, . 
\end{align}
We now sum over $i=1,\dots,N$, divide by $N$, and use the identity
\begin{align}
\frac{1}{N^{2}}\sum_{i,j=1}^{N} \nabla_{1}W(X_{t}^i,X_{t}^j) \cdot \nabla \varphi(X_{t}^i)=&\frac{1}{N}\sum_{i=1}^{N} \nabla_{1} W*\mu^{(N)}(X_{t}^i) \cdot \nabla \varphi(X_{t}^i)\\=& \langle \mu^{(N)}, \nabla_{1} W \star \mu^{(N)} \cdot \nabla \varphi \rangle \, ,    
\end{align}
to obtain
\begin{equation}
\dx \langle \mu^{(N)},\varphi \rangle = \langle \mu^{(N)}, \beta^{-1} \Delta \varphi - (\nabla V+\nabla_{1}W\star \mu^{(N)}) \cdot \nabla \varphi \rangle \dx t+\sqrt{2}\beta^{-\frac{1}{2}}\frac 1 N \sum_{i=1}^{N} \nabla \varphi(X^{i}_t) \cdot d B_{t}^i \nonumber
\end{equation}
Next, we insert the identity $\mu^{(N)}=\rho_{\beta}+N^{-\frac 1 2}\eta^{N}$, to deduce
\begin{align}
\dx \langle \eta^{N},\varphi \rangle =& N^{1/2}\underbrace{\langle \rho_\beta, \beta^{-1} \Delta \varphi - (\nabla V+\nabla_{1}W\star \rho_\beta) \cdot \nabla \varphi \rangle}_{=0} \dx t\\
&+\underbrace{\langle \eta^N, \beta^{-1} \Delta \varphi - (\nabla V+\nabla_{1}W\star \rho_\beta) \cdot \nabla \varphi \rangle+\langle \rho_\beta, \nabla_{1}W\star \eta^N \cdot \nabla \varphi \rangle}_{=\langle \eta^N,\mathcal{L}_{\rho_\beta}^*\varphi\rangle} \dx t \\
&+\underbrace{N^{-1/2}\langle \eta^N, \nabla_{1}W\star \eta^N \cdot \nabla \varphi \rangle}_{\langle R^N,\varphi\rangle}+\frac{\sqrt{2}}{\beta N^{\frac{1}{2}}}   \sum_{i=1}^{N} \nabla \varphi(X^{i}_t) \cdot d B_{t}^i  \nonumber
\end{align}
The first identity follows from the fact that $\rho_{\beta}$ is a steady state and the second follows from integration by parts and Fubini's theorem (using the symmetry of $W$).

{\sc Step} \arabic{CLT}.\label{CLTP11}\refstepcounter{CLT}
In this step, we will show that  the laws of $(\eta^{N})_{N \in \N}$ on  $C([0,T];H^{-m}(\T^d))$ for $m>m_{0}:=d/2 + 3$ are uniformly tight. To this end, we define a decomposition of $t \mapsto \eta^{N}(t)$ via the equality $\eta^{N}(t)=Y_{N}(t)+M_{N}(t)$, where
\begin{equation}
Y_{N}(t):=\eta^N(0)+\int_{0}^{t} \big ( \mathcal{L}_{\rho_{\beta}}\eta^{N}(r)+R^{N}(r) \big )\; \dx{r}.
\end{equation}
We claim that for $m>m_{0}$ and all $p \geq 1$ there exist constants $C,C_p>0$ such that
\begin{align}
    \E \pra*{\|Y^{N}(t_1)-Y^{N}(t_2)\|_{H^{-m}(\T^d)}^{2}} & \leq C|t_1-t_2|^{2} \label{s4}\\
    \E \pra*{\|M^{N}(t_1)-M^{N}(t_2)\|_{H^{-m}(\T^d)}^{p}} &\leq C_{p}|t_1-t_2|^{\frac p 2}.\label{s5} 
\end{align}
To obtain \eqref{s4}, first observe that at each fixed time we have
\begin{align}
    \|R_{N}\|_{H^{-m}(\T^d)}=&\|\nabla\cdot(\eta^N\nabla_{1}W\star(\mu^{(N)}-\rho_{\beta}))\|_{H^{-m}(\T^d)}\\\leq& \|\eta^N\nabla_{1}W\star(\mu^{(N)}-\rho_{\beta})\|_{H^{1-m}(\T^d)}\nonumber\\
    \leq& C_s\|\eta^{N}\|_{H^{1-m}(\T^d)}\|\nabla W\star(\mu^{(N)}-\rho_{\beta})\|_{W^{{m-1},\infty}(\T^d)}\\ \leq& C_{s}\|W\|_{W^{m,\infty}(\T^d)}\|\eta^{N}\|_{H^{1-m}(\T^d)}.  \label{Rbound}
\end{align}
Furthermore, note that the regularity of $V$, $W$, and $\rho_{\beta}$ implies the operator $\mathcal{L}_{\rho_{\beta}}$ is bounded from $H^{2-m}(\T^d)$ to $H^{-m}(\T^d)$.  Therefore, integrating in time, taking the second moment, and applying the Cauchy--Schwartz inequality we find
\begin{align}
    \E \pra*{\|Y^{N}(t_{1})-Y^{N}(t_{2})\|_{H^{-m}(\T^d)}^{2}}\le & \E\left(    \int_{t_{1}}^{t_{2}}\|\mathcal{L}_{\rho_{\beta}}\eta^N(r)\|_{H^{-m}(\T^d)}+\E\|R_{N}\|_{H^{-m}(\T^d)}\dx{r}\right)^2\\
    \leq& C|t_{2}-t_{1} | \int_{t_{1}}^{t_{2}}\E\|\eta^{N}\|_{H^{2-s}(\T^d)}^{2}+\E\|\eta^{N}\|_{H^{1-m}(\T^d)}^{2}\;\dx{r} \\
    \leq& C|t_{2}-t_{1}|^{2}\E\|\eta^{N}(0) \|_{H^{2-m}(\T^d)}^{2} \, , \nonumber
\end{align} 
where we used the stationarity of $\eta^{N}$ in the last step.  The inequality \eqref{s4} now follows immediately from the bound \eqref{L2control}, since our choice of $m$ implies $m-2>1+\frac{d}{2}$.  We now turn our attention to the estimate \eqref{s5}.  Note that for any smooth $\varphi$ it holds that 
\begin{align}
\E \pra*{ \abs*{ \langle M^{N}(t)-M^{N}(s), \varphi \rangle }^{p} }\leq& C \E\bigg (\frac{1}{N}\sum_{i=1}^{N}\int_{s}^{t}|\nabla \varphi(X_t^{i})|^{2}\dx t \bigg )^{\frac{p}{2}} \\\leq& C |t-s|^{\frac p 2} \|\nabla \varphi \|_{L^{\infty}(\T^d)}^{2}. \label{bdg}
\end{align}
We take $\varphi$ to be elements of a basis of $H^{m}_0(\T^d)$, the inequality \eqref{s5} holds as long as 
$$
\sum_{j\in\N}\|\nabla \varphi_j \|_{L^{\infty}(\T^d)}^{2}<\infty.
$$
The above estimate follows when we take $m>d/2+1$.
Finally, we note that by an argument entirely analogous to the one showing \eqref{s4} and \eqref{s5}, we can also show 
\begin{equation}
\sup_{N \in \N}\E\pra*{\sup_{t \in [0,T]} \|\eta^N(t)\|_{H^{-m}(\T^d)}^{2}}<\infty, \label{s3}
\end{equation}
by using that $M^{N}(0)=0$ and $Y^{N}(0)=\eta^{N}(0)$ together with \eqref{L2control}. Combining \eqref{s4}, \eqref{s5}, and \eqref{s3} we obtain the tightness of the laws of $(\eta^{N})_{N \in \N}$ as a consequence of \cite[Theorem 2.2]{flandoli1995martingale} and Chebyshev's inequality. Specifically, we use the embedding of $W^{1-\eps,2}([0,T];H^{-m+\eps}(\T^d))+W^{\frac{1}{2},p}([0,T]; H^{-m+\epsilon}(\T^d))$ into $C([0,T];H^{-m}(\T^d) )$ for $p>2$ and $\eps>0$ sufficiently small. 
 
 \medskip

{\sc Step} \arabic{CLT}.\label{CLTP12}\refstepcounter{CLT}
In light of Step 2 and the Skorokhod representation theorem, passing to a subsequence (which we do not relabel) we can find a new probability space $(\tilde{\mathcal{X}},\tilde{\mathcal{F}},\tilde{\mathbb{P}})$, a new sequence $(\tilde{\eta}^{N})_{N \in \N}$, and a limiting random variable $\eta$ such that $\tilde{\eta}^{N}$ is equal in law to $\eta^{N}$ for all $N \in \N$ and converges $\tilde{\P}$-a.s to $\eta$ in $C([0,T];H^{-m}(\T^d))$.  In this step, we claim that \eqref{Mart1} and \eqref{Mart2} hold true.  

To this end, for each $t>0$, we denote by $r_{t}$ the restriction operator from $C([0,T];H^{-m}(\T^d))$ to $C([0,t];H^{-m}(\T^d))$ and define a filtration $(\tilde{\mathcal{F}}_{t})_{t \geq 0}$ by letting $\tilde{\mathcal{F}}_{t}=\sigma(r_{t}\eta)$ for $t>0$, i.e. the sigma algebra generated by $r_t\eta$.  Recalling the definition \eqref{Mart0} of $t \mapsto M(t)$, we will show that that for all times $s<t$, functions $\varphi, \psi \in C^{\infty}(\T^{d})$, and bounded, continuous functions $\Gamma: C([0,s];H^{-m}(\T^d)) \mapsto \R$ it holds that
\begin{align}
    &\E \pra*{ \Gamma(r_{s}\eta)  \langle M(t)-M(s),\varphi \rangle  }=0. \label{s6} \\
    & \E \big[ \Gamma(r_{s}\eta) \big ( \langle M(t), \varphi \rangle \langle M(t), \psi \rangle - \langle  M(s), \varphi \rangle \langle M(s), \psi \rangle \label{s7}\\
    & \qquad -(t-s)\langle \rho_{\beta} \nabla \varphi , \nabla \psi \rangle \big ) \big]=0. \nonumber
\end{align}
By the definition of conditional expectation and Egorov's theorem, \eqref{s6} and \eqref{s7} imply \eqref{Mart1} and \eqref{Mart2} hold true with respect to the filtration $(\tilde{\mathcal{F}}_{t})_{t \geq 0}$.  To prove \eqref{s6} and \eqref{s7}, define $t \mapsto \tilde{M}^{N}(t)$ as in Step 2 but with $\tilde{\eta}_{N}$ in place of $\eta^N$.
 
Since $\tilde{\eta}^{N}$ and $\eta^{N}$ are equal in law, \eqref{s1} implies that
\begin{align}
    &\E \big [ \Gamma(r_{s}\tilde{\eta}_{N})  \langle \tilde{M}^{N}(t)-\tilde{M}^{N}(s),\varphi \rangle  \big]=0. \label{s9} \\
    & \E \big [\Gamma(r_{s} \tilde{\eta}_{N} ) \big ( \langle \tilde{M}_{N} (t), \varphi \rangle \langle \tilde{M}_{N} (t), \psi \rangle - \langle  \tilde{M}_{N} (s), \varphi \rangle \langle \tilde{M}_{N}(s), \psi \rangle \label{s10} \\
    &\qquad \qquad-(t-s)\langle \rho_{\beta} \nabla \varphi , \nabla \psi \rangle \big ) \big]=0. \nonumber
\end{align}
 By a calculation similar to \eqref{Rbound}, it follows that for each $t \leq T$
\begin{equation}
\bigg \| \int_{0}^{t}\tilde{R}_{N}(s)ds \bigg \|_{H^{-(m+1)}(\T^d)} \leq C N^{-1/2}\|\tilde{\eta}_{N}\|_{C([0,T];H^{-m}(\T^d) )}^{2},   
\end{equation}  
which converges to zero $\tilde{P}$-a.s.  As a consequence, we obtain $\langle \tilde{M}^{N}(t), \varphi \rangle$ converges  $\tilde{P}$-a.s. to $\langle M(t), \varphi \rangle$ as a consequence of the a.s. convergence of $\tilde{\eta}^{N}$ to $\eta^{\infty}$.  In addition,  for all $t>0$, the sequence $\langle \tilde{M}^{N}(t), \varphi \rangle$ is uniformly bounded in $L^{p}(\tilde{\mathcal{X}})$ for all $p \geq 1$ as a consequence of the equality in law of $\tilde{\eta}^{N}$ and $\eta^{N}$ and the estimate \eqref{bdg}.  Using the Vitali convergence theorem, we may pass to the limit in \eqref{s9} and \eqref{s10} to obtain \eqref{s6} and \eqref{s7} as desired.  
\medskip

{\sc Step} \arabic{CLT}.\label{CLTP13}\refstepcounter{CLT}
In light of Step 3 and Lemma \ref{martingalelem}, every subsequence of $\eta^N$ has a further subsequence which converges to a stationary mild solution to \eqref{SPDE1} on some probability space.  Note that $\eta$ inherits stationarity from $\tilde{\eta}^{N}$ in the $N \to \infty$ limit.
Hence, it suffices to show that all limit points induce the same law on $C([0,T];H^{-m}(\T^d))$.  In light of Remark \ref{rem:unique}, the problem further reduces to showing that the initial distributions are the same.  However, if $\eta$ satisfies \eqref{eq:Mild} and is stationary, we can explicitly check that its law is a Gaussian on $H^{-m}(\T^d)$  for all $t \geq 0$. This follows from~\cite[Theorem 5.22 and Proposition 5.23]{Hairer09},~\eqref{eq:trQ}, and the fact that for all $f \in H^{-m}(\T^d)$, we have 
\begin{equation}
\lim_{t \to \infty}\| e^{t \mathcal{L}_{\rho_\beta}} f\|_{H^{-m}} =0 \, .
\end{equation}
The fact that this holds true follows from the coercivity assumption and the fact that the semigroup $e^{t\mathcal{L}_{\rho_\beta}}$ is smoothing, i.e. it maps $H^{-m}(\T^d)$ to $L^2(\T^d)$ for all $t>0$. 
\end{proof}

We finish this section with the implication of the coercivity property, that is used in the proof to show the Hilbert-Schmidt property of the appropriate operators.
\begin{lemma}\label{lem:gap}
Assume that $\mathcal{L}_{\rho_\beta}$ satisfies
$$
-\langle \mathcal{L}_{\rho_\beta}\phi,\phi\rangle\ge c\|\nabla \phi\|_{L^2(\T^d)}^2
$$
for some $c>0$, then
$$
\int_0^\infty \|\nabla e^{t\mathcal{L}^*_{\rho_\beta}}\phi\|_{L^2(\T^d)}^2\;\dx{t}\le \frac{\|\phi\|_{L^2}^2}{2c}.
$$
\end{lemma}
\begin{proof}
We use the short hand notation $\phi_t=e^{t\mathcal{L}^*_{\rho_\beta}}\phi$. We then have
$$
\partial_t \phi_t -\mathcal{L}_{\rho_\beta}^*\phi_t=0 \, .
$$
Multiplying by $\phi_t$ and integrating, we obtain the identity
$$
\frac12\frac{\dx}{\dx{t}}\|\phi_t\|^2_{L^2(\T^d)}-\langle \mathcal{L}_{\rho_\beta}\phi_t,\phi_t\rangle=0.
$$
Applying the coercivity bound, we obtain
$$
\frac{\dx}{\dx{t}} \|\phi_t\|_{L^2(\T^d)}^2+2c\|\nabla \phi_t\|_{L^2}^2 \le 0.
$$
Integrating in time we obtain
$$
2c \int_0^\infty \|\nabla \phi_t\|_{L^2(\T^d)}^2\;\dx{t}\le \|\phi_0\|_{L^2}^2=\|\phi\|_{L^2(\T^d)}^2 \, ,
$$
which is the desired estimate.
\end{proof}

\section*{Acknowledgements}
The authors would like to thank Martin Hairer and Luigia Ripani for useful discussions during the course of this work. G.A.P. was partially supported by the EPSRC through the grant number EP/P031587/1 and by JPMorgan Chase $\&$ Co under a J.P. Morgan A.I. Research Award 2019. M.G.D. was partially supported by CNPq-Brazil (\#308800/2019-2) and Instituto Serrapilheira.

\bibliographystyle{imsart-number}

\bibliography{ref}

\end{document}